\documentclass[11pt,reqno,tbtags,a4paper]{amsart}
\usepackage{amssymb}
\usepackage{url}
\usepackage[square,numbers]{natbib}
\bibpunct[, ]{[}{]}{;}{n}{,}{,}

\title
{On edge exchangeable random graphs}

\date{21 February, 2017}

\author{Svante Janson}
\address{Department of Mathematics, Uppsala University, PO Box 480,
SE-751~06 Uppsala, Sweden}
\email{svante.janson@math.uu.se}
\urladdr{http://www.math.uu.se/svante-janson}

\subjclass[2010]{}
\subjclass[2010]{05C80;  05C65}

\numberwithin{equation}{section}

\renewcommand\le{\leqslant}
\renewcommand\ge{\geqslant}

\allowdisplaybreaks

\theoremstyle{plain}
\newtheorem{theorem}{Theorem}[section]
\newtheorem{lemma}[theorem]{Lemma}
\newtheorem{proposition}[theorem]{Proposition}
\newtheorem{corollary}[theorem]{Corollary}

\theoremstyle{definition}
\newtheorem{example}[theorem]{Example}
\newtheorem{definition}[theorem]{Definition}
\newtheorem{problem}[theorem]{Problem}
\newtheorem{remark}[theorem]{Remark}

\theoremstyle{remark}

\newenvironment{romenumerate}[1][-10pt]{
\addtolength{\leftmargini}{#1}\begin{enumerate}
 \renewcommand{\labelenumi}{\textup{(\roman{enumi})}}%
 }{\end{enumerate}}

\newcounter{oldenumi}
{\setcounter{oldenumi}{\value{enumi}}
\begin{romenumerate} \setcounter{enumi}{\value{oldenumi}}}
{\end{romenumerate}}

\newcounter{thmenumerate}

\newcounter{xenumerate}   

\newcommand\xfootnote[1]{\unskip\footnote{#1}$ $} 

\newcommand\pfitemx[1]{\par#1:}
\newcommand\pfitemref[1]{\pfitemx{\ref{#1}}}

\newcommand\step[2]{\smallskip\noindent#1 \emph{#2}:}
\newcounter{steps}
\newcommand\stepx{\smallskip\noindent\refstepcounter{steps}%
 \emph{Step \arabic{steps}. }\noindent}

\newcommand{\refT}[1]{Theorem~\ref{#1}}
\newcommand{\refC}[1]{Corollary~\ref{#1}}
\newcommand{\refL}[1]{Lemma~\ref{#1}}
\newcommand{\refR}[1]{Remark~\ref{#1}}
\newcommand{\refS}[1]{Section~\ref{#1}}
\newcommand{\refSS}[1]{Section~\ref{#1}}
\newcommand{\refP}[1]{Proposition~\ref{#1}}
\newcommand{\refD}[1]{Definition~\ref{#1}}
\newcommand{\refE}[1]{Example~\ref{#1}}

\newcommand{\refStep}[1]{Step~\ref{#1}}

\newcommand\REM[1]{{\raggedright\texttt{[#1]}\par\marginal{XXX}}}
\newcommand\XREM[1]{\relax}

\begingroup
  \divide\count255 by 60
  \multiply\count255 by -60
  \advance\count255 by \time
  \ifnum \count255 < 10 \xdef\klockan{\the\count1.0\the\count255}
  \else\xdef\klockan{\the\count1.\the\count255}\fi
\endgroup

\newcommand{\sumi}{\sum_{i=1}^\infty}

\newcommand{\prodin}{\prod_{i=1}^n}

\newcommand\set[1]{\ensuremath{\{#1\}}}
\newcommand\bigset[1]{\ensuremath{\bigl\{#1\bigr\}}}

\newcommand\xpar[1]{(#1)}
\newcommand\bigpar[1]{\bigl(#1\bigr)}
\newcommand\Bigpar[1]{\Bigl(#1\Bigr)}

\newcommand\xcpar[1]{\{#1\}}

\newcommand\bigabs[1]{\bigl|#1\bigr|}
\newcommand\Bigabs[1]{\Bigl|#1\Bigr|}

\def\rompar(#1){\textup(#1\textup)}    

\newcommand\xpqfrac[2]{(#1)/(#2)}

\newcommand\Bigparfrac[2]{\Bigpar{\frac{#1}{#2}}}

\def\xexp(#1){e^{#1}}
\newcommand\ceil[1]{\lceil#1\rceil}

\newcommand\Bigceil[1]{\Bigl\lceil#1\Bigr\rceil}
\newcommand\floor[1]{\lfloor#1\rfloor}

\newcommand\ntoo{\ensuremath{{n\to\infty}}}
\newcommand\Ntoo{\ensuremath{{N\to\infty}}}

\newcommand\mtoo{\ensuremath{{m\to\infty}}}

\newcommand\ttoo{\ensuremath{{t\to\infty}}}
\newcommand\xtoo{\ensuremath{{x\to\infty}}}

\newcommand\norm[1]{\|#1\|}

\newcommand\punkt{.\spacefactor=1000}    
\newcommand\iid{i.i.d\punkt}    
\newcommand\ie{i.e\punkt}
\newcommand\eg{e.g\punkt}

\newcommand\cf{cf\punkt}
\newcommand{\as}{a.s\punkt}
\newcommand{\aex}{a.e\punkt}

\newcommand{\tend}{\longrightarrow}
\newcommand\dto{\overset{\mathrm{d}}{\tend}}
\newcommand\pto{\overset{\mathrm{p}}{\tend}}
\newcommand\asto{\overset{\mathrm{a.s.}}{\tend}}

\newcommand\op{o_{\mathrm p}}
\newcommand\Op{O_{\mathrm p}}

\newcommand\bbR{\mathbb R}

\newcommand\bbN{\mathbb N}
\newcommand\bbNo{\mathbb N_0}
\newcommand\bbNow{\bbN\cup\set{0,-1}}
\newcommand\bbNw{\bbN\cup\set{-1}}

\newcommand\bbZ{\mathbb Z}

\newcounter{CC}
\newcommand{\CC}{\stepcounter{CC}\CCx} 
\newcommand{\CCx}{C_{\arabic{CC}}}     
\newcommand{\CCdef}[1]{\xdef#1{\CCx}}     
\newcommand{\CCname}[1]{\CC\CCdef{#1}}    
\newcommand{\CCreset}{\setcounter{CC}0} 
\newcounter{cc}
\newcommand{\cc}{\stepcounter{cc}\ccx} 
\newcommand{\ccx}{c_{\arabic{cc}}}     
\newcommand{\ccreset}{\setcounter{cc}0} 

\newcommand\E{\operatorname{\mathbb E{}}}
\renewcommand\P{\operatorname{\mathbb P{}}}

\newcommand\Var{\operatorname{Var}}
\newcommand\Cov{\operatorname{Cov}}

\newcommand\Exp{\operatorname{Exp}}
\newcommand\Po{\operatorname{Po}}

\newcommand\Be{\operatorname{Be}}

\newcommand\ga{\alpha}

\newcommand\gd{\delta}

\newcommand\gf{\varphi}
\newcommand\gam{\gamma}
\newcommand\gG{\Gamma}

\newcommand\gl{\lambda}

\newcommand\gO{\Omega}
\newcommand\gs{\sigma}

\newcommand\gss{\sigma^2}
\newcommand\gth{\theta}
\newcommand\eps{\varepsilon}

\newcommand\cA{\mathcal A}

\newcommand\cE{\mathcal E}
\newcommand\cF{\mathcal F}

\newcommand\cL{{\mathcal L}}

\newcommand\cS{{\mathcal S}}

\newcommand\cW{\mathcal W}

\newcommand\tG{\tilde G}

\newcommand\tP{\widetilde P}

\newcommand\ett[1]{\boldsymbol1\xcpar{#1}} 
\newcommand\indic[1]{\boldsymbol1_{\xcpar{#1}}}

\newcommand\etta{\boldsymbol1}

\newcommand\qw{^{-1}}

\newcommand\oi{\ensuremath{[0,1]}}

\newcommand\ooo{[0,\infty)}

\newcommand\setoi{\set{0,1}}

\newcommand\dd{\,\mathrm{d}}

\newcommand\lhs{left-hand side}
\newcommand\rhs{right-hand side}

\newcommand\xij{_{ij}}

\newcommand\qx{^*}
\newcommand\csx{\cS\qx}
\newcommand\csxx[1]{\cS^{\lor#1}}
\newcommand\csxn{\csxx{n}}
\newcommand\csxz{\cS^{**}}
\newcommand\cszz[1]{\cS^{\land#1}}
\newcommand\cszn{\cszz{n}}

\newcommand\GGq{\Gq^*}
\newcommand\Gq{G}
\newcommand\GG{\G^*}
\newcommand\G{\tilde{G}}
\newcommand\GGm{\GGq_m}
\newcommand\Gm{\Gq_m}
\newcommand\GGt{\GG_t}
\newcommand\Gt{\G_t}
\newcommand\ggndd{\ensuremath{\GG(n,\ddn)}}
\newcommand\nnn{_{n=1}^\infty}
\newcommand\mmm{_{m=1}^\infty}
\newcommand\iii{_{i=1}^\infty}
\newcommand\gsf{$\gs$-finite}
\newcommand\bmu{\bar\mu}
\newcommand\pset[1]{(\set{#1})}
\newcommand\ddn{(d_i)_{i=1}^n}
\newcommand\hP{\hat P}
\newcommand\GEM{\mathrm{GEM}}
\newcommand\PD{\mathrm{PD}}
\newcommand\DIR{\operatorname{Dir}}
\newcommand\tgs{\tilde\gs}
\newcommand\tgss{\tilde\gs^2}
\newcommand\bI{\bar I}
\newcommand\bIx{\bar I'}
\newcommand\KK{X}
\newcommand\ellN{N}
\newcommand\cn[1]{\norm{#1}\cut}

\newcommand\cut{_{\square}}
\newcommand\dcut{\delta_{\square}}

\newcommand\norml[1]{\norm{#1}_{L^1}}
\newcommand\mpp{measure-preserving}
\newcommand\ellt{\ell(t)}
\newcommand\xxH{H}
\newcommand\Beta{\operatorname{Beta}}
\newcommand\gGq{\gG_{\mathsf{half}}}
\newcommand\NN{N}
\newcommand\NNx{N^*}
\newcommand\gek{_{\ge k}}
\newcommand\qgam{^{1/\gam}}

\newcommand\tlogt{t/\log t}
\newcommand\tlogtt{t/\log^2t}
\newcommand\bbRp{\bbR_+}
\newcommand\bG{\bar G}
\newcommand\togp{\to_{\mathsf{GP}}}
\newcommand\togs{\to_{\mathsf{GS}}}
\newcommand\hwt{\hat W_t}
\newcommand\tgg{t^{1/2\gam}}
\newcommand\ZZ{(Z_{kl})_{k,l}}
\newcommand\ZZZ{(Z'_{kl})_{k,l}}
\newcommand\bS{\bar S}
\newcommand\Sx{S^*}
\newcommand\XX{\mathbf{X}}

\newcommand{\Polya}{P\'olya}

\newcommand{\Lovasz}{Lov\'asz}

\begin{document}

\begin{abstract} 
We study a recent model for edge exchangeable random graphs introduced by
Crane and Dempsey; in particular we study asymptotic
properties of the random simple graph
obtained by merging multiple edges.
We study a number of examples, and show
that the model can produce dense, sparse and extremely sparse random graphs.
One example yields a power-law degree distribution.
We give some examples where the random graph is dense and converges a.s.\ in the
sense of graph limit theory, 
but also an example where a.s.\ every graph limit is the limit of some
subsequence. Another example is sparse and yields convergence to a
non-integrable generalized graphon defined on $(0,\infty)$.
\end{abstract}

\maketitle


\section{Introduction}\label{S:intro}

A model for edge exchangeable random graphs and
hypergraphs was recently introduced by
\citet{CraneD-edge,CraneD-relational}, 
who also gave a representation theorem
showing that every infinite
edge exchangeable random hypergraph can be constructed by this model.
An 
equivalent model, using somewhat different formulations,
was given by \citet{BroderickCai} and \citet{CampbellCaiBroderick}, see
\refR{Rcai}. 

The idea of the model is that random \iid{} edges,
with an arbitrary distribution, are added to a fixed
vertex set; see
\refS{Sdef} for a detailed definition
(slightly modified but equivalent to the original definition). 

The general model defines a random hypergraph.
In the present paper, we concentrate on the graph case, although we state
the definitions in \refS{Sdef} more generally for hypergraphs.

Since edges can be repeated, 
the model defines a random multigraph, but this can as always be
reduced to a random simple graph by identifying parallel edges and
deleting loops. 
Typically, many of the edges will be repeated many times, see \eg{}
\refR{Rmany}, and thus the multigraph and the simple graph versions can be
expected to be quite different.
Both versions have interest and potential, possibly
different, applications,
and we consider both 
versions.
Previous papers
concentrate on the multigraph version;
in contrast and as a complement, in the present paper we
study mainly the simple graph version.

The model is, as said above, based on an arbitrary distribution of edges.
Different choices of this distribution  can give a wide range of 
different types of random graphs, and the main purpose of the paper is to
investigate the types of random graphs that may be created by this model;
for this purpose we 
give some general results on the numbers of vertices and edges, and a number
of examples ranging from dense to very sparse graphs.
The examples show that the model can produce very different graphs. 
In some dense examples we show that the random graphs converge in the sense
of graph limit theory. However, that is not always the case, and we even
give a chameleon example (\refT{Tchameleon})
that has every graph limit as the limit of some
subsequence. 
We give also a sparse example (\refE{Epower}) with a power-law degree
distribution and convergence to a generalized graphon in the sense of
\citet{VR2}. 

An important tool in our investigations is a Poisson version of the
construction by \citet{CraneD-relational}, see
\refSS{SSdefPo}, 
which seems interesting also in its own right.

After some preliminaries in Sections \ref{Snot}--\ref{Slim}, we give the
definitions of the random hypergraphs in detail in \refS{Sdef}.
The graph case is discussed further in \refS{Sgraphs}.
\refS{Snumber} studies the numbers of vertices and edges in the graphs.
\refS{Srank1} considers an important special case of the model, called
\emph{rank 1}; we study two multigraph examples previously considered by
\citet{CraneD-edge} and \citet{Pittel} and show that they are of this
type.

The remaining sections consider various examples of the simple graph
version,
with dense examples in \refS{Sdense}, and sparse examples in Sections
\ref{Ssparse} and \ref{SXsparse}.
Finally, we give some tentative conclusions in \refS{Sconclusion}.

\section{Some notation} \label{Snot}

In general, we allow hypergraphs to have multiple edges; 
we sometimes (but usually not) say
multihypergraph for emphasis.
Moreover, the edges in a hypergraph may  have repeated vertices,
\ie, the edges are in general multisets of vertices, see \refR{Rmulti}.
An edge with repeated vertices is called a loop.
A simple hypergraph is a hypergraph without multiple edges and loops.
(\emph{Warning}: 
different authors give different meanings to 'simple hypergraph'.) 

The vertex and edge sets of a multigraph $G$ are denoted by $V(G)$ and
$E(G)$,
and the numbers of vertices and edges by $v(G):=|V(G)|$ and $e(G):=|E(G)|$.

$f(x)\sim g(x)$ means $f(x)/g(x)\to1$ 
(as $x$ tends to some limit, \eg{} $x\to\infty$).
We also use $v\sim w$ for adjacency of two vertices $v$ and $w$ in a given
graph, and $X\sim\cL$ meaning that the random variable $X$ has distribution
$\cL$; there should not be any risk of confusion between these (all
standard) uses of $\sim$.

$f(x)\asymp g(x)$ for two non-negative functions 
or sequences $f(x)$ and $g(x)$ (defined on some
common set $S$) means that $f/g$ and $g/f$ both are bounded; equivalently,
there exist constants $c,C>0$ such that $cg(x)\le f(x)\le C g(x)$
for every $x\in S$. 
$f(x)\asymp g(x)$ as \xtoo{} means that 
$f(x)\asymp g(x)$ for 
$x$ in some interval $[x_0,\infty)$.

We use 'increasing' (for a function or a sequence) in its weak sense 
\ie, $x\le y\implies f(x)\le f(y)$,
and similarly with  'decreasing'.

$x\land y$ is  $\min\set{x,y}$
and $x\lor y$ is $\max\set{x,y}$.

$\bbN:=\set{1,2,\dots}$ and $\bbNo:=\set{0,1,2,\dots}$.
$[n]:=\set{1,\dots,n}$.

If $\mu$ is a measure on a set $\cS$, then $\norm\mu:=\mu(\cS)\le\infty$.

$\Exp(\gl)$ denotes 
the exponential distribution with \emph{rate} $\gl$,
\ie, the first point in a Poisson process with rate $\gl$; this is thus the
exponential distribution with mean $1/\gl$.
For convenience we extend this to $\gl=0$: $X\sim\Exp(0)$ means $X=+\infty$ a.s.

We say that a sequence $G_n$ of simple graphs with $v(G_n)\to\infty$
is \emph{dense} if $e(G_n)\asymp v(G_n)^2$,
\emph{sparse} if $e(G_n)=o(v(G_n)^2)$,
and \emph{extremely sparse} if $e(G_n)\asymp v(G_n)$ 
as \ntoo,
and similarly for a family $G_t$ of
graphs with a continuous parameter.

We let $C,c,C_1,c_1,\dots$ denote various unspecified positive constants.

\section{Some preliminaries on graph limits, graphons and cut metric}\label{Slim}

We recall some basic facts on graph limits and graphons. 
For further details, see 
\eg{} \cite{BCLSV1,BCLSV2}, \cite{SJ209} and the comprehensive book
\cite{Lovasz}. 

A (standard) \emph{graphon} is a symmetric measurable function
$W:\gO\times\gO\to\oi$,  where $\gO=(\gO,\cF,\mu)$ is a probability space.
($\gO$ may without loss of generality be taken as $\oi$ with Lebesgue
measure, but it is sometimes convenient to use other probability spaces too.)

If $\gf:\gO_1\to\gO_2$ is a \mpp{} map between two probability spaces
$\gO_1$ and $\gO_2$, and $W$ is a graphon  on $\gO_2$, then
$W^\gf(x,y):=W(\gf(x),\gf(y))$ is a graphon on $\gO_1$ called the
\emph{pull-back} of $W$.

If $W$ is an integrable function on $\gO^2$, then its \emph{cut norm} is
\begin{equation}
  \cn{W}:=\sup\Bigabs{\int_{T\times U}W(x,y)\dd\mu(x)\dd\mu(y)},
\end{equation}
taking the supremum over all measurable sets $T,U\subseteq\gO$.

For two graphons $W_1$ and $W_2$, defined on probability spaces $\gO_1$ and
$\gO_2$, their \emph{cut distance} is defined as
\begin{equation}
  \dcut(W_1,W_2) = \inf_{\gf_1,\gf_2}\cn{W_1^{\gf_1}-W_2^{\gf_2}},
\end{equation}
taking the infimum over all pairs $(\gf_1,\gf_2)$ of \mpp{} maps
$\gf_j:\gO\to\gO_j$ defined on some common probability space $\gO$.

Two graphons $W_1$ and $W_2$ are \emph{equivalent} if $\dcut(W_1,W_2)=0$.
Note that a graphon $W$ and any pullback $W^{\gf}$ of it are equivalent. For
characterizations of equivalent graphons, see \cite{BCL:unique} and
\cite[Section 8]{SJ249}.
The cut distance $\dcut$ can be regarded as a metric on the set $\cW$ of
equivalence classes of graphons, and makes $\cW$ into a compact metric space.

A \emph{graph limit} can be identified with an equivalence class of
graphons, so we can regard $\cW$ as the space of graph limits. 
Thus, every graphon defines a graph limit, and every
graph limit is represented by some graphon, but this graphon is unique only
up to equivalence.

For every finite graph $G$, there is a corresponding graphon $W_G$ that can
be defined by taking $\gO=V(G)$ with the uniform probability measure
$\mu\set{i}=1/v(G)$ for every $i\in V(G)$ and letting $W_G(i,j)=\indic{i\sim
  j}$; thus $W_G$ equals the adjacency matrix of $G$, regarded as a function
$V(G)^2\to\setoi$. ($W_G$ is often defined as an equivalent
graphon on $\oi$; for us this makes no difference.)
We identify $G$ and $W_G$ when convenient, and write for example
$\dcut(G,W)=\dcut(W_G,W)$ for a graph $G$ and a graphon $W$.
\xfootnote{
Usually I keep a distiction between graphs and graphons (and graph limits);
this is easiest done by identifying a graph $G$ with the pair
$(W_G,1/v(G))\in \cW\times\oi$ (and a graphon $W$ with $(W,0)$), see also
\cite{SJ209}. In the present paper, this point of view is not needed.
}

\begin{remark}\label{Rblowup}
Let $G$ be a finite graph.
A \emph{blow-up} $G^*$ of $G$ is the graph obtained by taking,
for some integer $m\ge1$,
the vertex set $V(G^*)=V(G)\times[m]$ with $(v,i)\sim(w,j)$ in $G^*$ if and
only if $v\sim w$ in $G$.
Then, $W_{G^*}$ is a pull-back of $W_G$ (for $\gf:V(G^*)\to V(G)$ the
natural projection), and thus 
$\dcut(G^*,G)=\dcut(W_G,W_{G^*})=0$.
Hence the graphs $G$ and $G^*$, which are different (if $m>1$) are
equivalent when regarded as graphons.
\end{remark}

There are several, quite different but nevertheless equivalent, 
ways to define convergence of a sequence of graphs, see \eg{}
\cite{BCLSV1,BCLSV2,SJ209,Lovasz}.
For our purposes it suffices to know that
a sequence $G_n$ with $v(G_n)\to\infty$ is \emph{convergent} if 
and only if there exists a graphon $W$ such that $\dcut(G_n,W)\to0$ as \ntoo.
We then say that $G_n$ converges to $W$, or to the corresponding graph limit. 

\subsection{Sparse graph limits and graphons}\label{Sgraphex}

The standard graphons defined above are appropriate for dense graphs.
In \refS{Ssparse} we consider an example with sparse graphs, where we
can show that the graphs $\Gt$ and $\Gm$ converge in a suitable sense to
a more general type of graphons defined by \citet{VR} as a symmetric
measurable function $W:\bbRp^2\to\oi$, see also \cite{BCCH16}.

\citet{VR2} defined two notions $\togp$ and $\togs$
of convergence for such general graphons on $\bbRp$ 
(and the even more general graphexes)
based on convergence in distribution of the corresponding random graphs.
Given a graphon $W$ on $\bbRp$, we define, see \cite{CaronFox,BCCH16,VR},
for each $r\ge0$ 
an unlabelled 
random graph $G_r(W)$  by taking
a Poisson process with intensity $r$ on $\bbRp$,  
regarded as a random point set  $\set{\eta_i}_{i\ge 1}$;
given a realization of this
Poisson process, we let $\bG_r(W)$ be the graph with vertex set $\bbN$ and an
edge $ij$ with probability $W(\eta_i,\eta_j)$ for every pair $(i,j)$ with
$i<j$;
finally we let $G_r(W)$ be the graph obtained by deleting all isolated
vertices from $\bG_r(W)$, and then ignoring labels.
(We assume that $W$ is such that $G_r(W)$ is \as{} finite, see \cite{VR} for
precise conditions.)
We can now define $W_n\togp W$ as meaning $G_r(W_n)\dto G_r(W)$ for each
$r<\infty$, 
see further \cite{VR2} and \cite{SJ317}. 

Furthermore,
the random graphs $G_r(W)$ are naturally coupled for different $r$ and form
an increasing graph process $(G_r(W))_{r\ge0}$. Let $(G_{\tau_k}(W))_k$ be
the sequence of different graphs that occur among $G_r(W)$ for $r\ge0$.
Then $W_n\togs W$ if
$(G_{\tau_k}(W_n))_k\dto (G_{\tau_k}(W))_k$; again
see further \cite{VR2} and \cite{SJ317}. 

\begin{remark}
We have here given the version of $G_r(W)$ without
loops;
more generally, one can allow $i=j$ in the construction and thus allow
loops. The loopless case considered here then is obtained by imposing
$W(x,x)=0$ for $x>0$. Hence, for the version with loops, \refT{Tgraphex} below
still holds, provided we redefine $W$ to be 0 on the diagonal.
\end{remark}

In the context of graphons on $\bbRp$, we define 
a modification of the graphon $W_G$ defined above for a finite graph $G$
(which we may assume has vertices labelled $1,\dots,v(G)$):
For every 
$s>0$, we let the \emph{stretched graphon}
$W_{G,s}$ be the function $\bbRp^2\to\setoi$ given by
$W_{G,s}(x,y):=W_G(\ceil{sx},\ceil{sy})$, where 
$W_G(i,j)=\indic{i\sim j}$ is as
above, extended by $W_G(i,j)=0$ when $i\lor j>v(G)$. 
Hence, every vertex in $G$ corresponds to an interval of length $1/s$ in the
domain of $W_{G,s}$.

\section{Constructions of random hypergraphs}\label{Sdef}
In this section, we define the random hypergraphs.
We give several versions; we define both multihypergraphs and simple
hypergraphs, 
and we give both the original version with a fixed number of edges and a
Poisson version.

In later sections we consider only the graph case, but we give the
definitions here in greater generality.

  Note that the edge exchangeable random hypergraphs constructed  
here are quite different from the vertex exchangeable graphs 
in
\eg{} \cite{BCLSV1,BCLSV2,Lovasz,CaronFox,VR,BCCH16}, see \refSS{SSvex}.

We begin with some preliminaries.

Let $(\cS,\cF)$ be a measurable space, for convenience usually denoted
simply by $\cS$.
To avoid uninteresting technical complications, we assume that $\cS$ is a
Borel space, \ie, isomorphic to a Borel subset of a complete separable
metric space with its Borel $\gs$-field.

Let $\csx$ be the set of all finite non-empty multisets of points in $\cS$.
We can regard a multiset with $n$ elements as an equivalence class of
sequences $(x_1,\dots,x_n)\in\cS^n$, where two such sequences are equivalent
if one is a permutation of the other. Denoting this equivalence relation by
$\cong$ 
and the set of multisets of $n$ elements in $\cS$ by $\csxn$, we thus have
$\csxn=\cS^n/{\cong}$ and $\csx=\bigcup_{n=1}^\infty\csxn$.
Note that $\csxn$ and $\csx$ are Borel spaces.
(One way to see this is to recall that every Borel space is isomorphic to a
Borel subset of $\oi$. We may thus assume that $\cS\subseteq\oi$, and then we
can redefine $\csxn$ as $\set{(x_1,\dots,x_n)\in\cS^n:x_1\le\dots\le x_n}$,
which is a Borel subset of $\oi^n$.)

\begin{remark}
   Definitions \ref{D1} and \ref{D2} below use a probability
measure $\mu$ to define
the random (hyper)graphs. In general, this measure may be a random measure,
and then the constructions should be interpreted by conditioning on $\mu$, \ie,
by first sampling $\mu$, and then using the obtained measure throughout the
construction. In other words, the distribution of the random hypergraphs
constructed by a random measure $\mu$ is a mixture of the distributions
given by deterministic $\mu$. 
For convenience, and because most examples will be with deterministic $\mu$,
we usually tacitly assume that $\mu$ is deterministic; results in the
general case with random $\mu$ then follow by conditioning on $\mu$.
(See \refR{Rpo} for a typical example, where this for once is stated
explicitly.) 
\end{remark}

\subsection{Random hypergraphs with a given number of edges}
\label{SSdefn}
We give a minor modification of the original definition by
\citet{CraneD-edge,CraneD-relational}; we will see att the end of this
subsection that our definition is equivalent to the original one.

\begin{definition}\label{D1}
  Given a Borel space $\cS$ and
a probability measure $\mu$  on $\csx$,  
define a sequence of finite random (multi)hypergraphs $(\GGm)\mmm$
as follows. 
Let
$Y_1,Y_2,\dots$ be an infinite sequence of \iid{} random multisets with
the distribution $\mu$, and let for every 
$m\in\bbN$ the hypergraph $\GGm$
be the multihypergraph with edges $Y_1,\dots,Y_m$
and vertex set $\bigcup_{i=1}^m Y_i$, \ie, the vertex set spanned by the
edges. (Thus there are no isolated vertices in $\GGm$.)

We also similarly define the infinite (multi)hypergraph $\GGq_\infty$ having 
edges $(Y_i)\iii$.

The edges in $\GGm$ may be repeated, so $\GGm$ is in general a random
multihypergraph. We define $\Gm$ as the simple hypergraph obtained by merging
each set of parallel edges in $\GGm$ to a single edge and deleting loops; thus
$\Gm$ has edge set $E(\Gm)=\set{Y_i:i\le m,\,Y_i\text{ not loop}}$ 
and vertex set
$V(\Gm)=\bigcup_{Y_i\in E(\Gm)} Y_i\subseteq V(\GGm)$.
\end{definition}

Note that $\GGq_1\subset \GGq_2\subset \dots$,
and thus $\Gq_1\subseteq \Gq_2\subseteq \dots$, 
\ie,  $(\GGm)_m$ and $(\Gm)_m$ are increasing
sequences of random hypergraphs.

\begin{remark}\label{Rmulti}
We follow \cite{CraneD-relational} and allow for increased generality
 $Y_i$ to be a multiset
(see \eg{} the examples in \refS{Srank1});
thus the edges in $\GGm$ and $\Gm$ are multisets and 
may contain repeated  vertices.
If we choose $\mu$ with support in the set 
$\csxz:=\bigcup_{n=1}^\infty \cszn\subset\csx$ of finite subsets of $\cS$, where
$\cszn\subset\csxn$ is the set of subsets of $\cS$ with $n$ distinct elements,
then the edges in $\GGm$ and $\Gm$ are ordinary
sets of vertices (\ie, without repeated vertices).
(This is commonly assumed in the definition of hypergraphs.)

In particular, if $\mu$ has support in
$\cszz2=\set{\set{x,y}:x,y\in\cS,\,x\neq y}$, then
$\GGm$ is a multigraph without loops, and $\Gm$ is a simple graph with 
$V(\Gm)=V(\GGm)$.
\end{remark}

The construction above yields hypergraphs with vertices labelled by elements
of $\cS$. We (usually)
ignore these labels and regard $\GGm$ and $\Gm$ as
unlabelled hypergraphs.

\begin{remark}\label{RdeFinetti}
  We usually also ignore the labels on the edges.
If we keep the labels $i$ on the edges $Y_i$, then the distribution of
$\GGm$ is obviously edge exchangeable, \ie, invariant under permutations
of these edge labels, because $(Y_i)_i$ is an \iid{} sequence. 
Conversely, as shown by \citet[Theorem 3.4]{CraneD-relational}, 
every infinite edge exchangeable hypergraph is a mixture of random
hypergraphs $\GGq_\infty$, \ie, it can be constructed as above using a
random measure $\mu$. 
In the present formulation, the proof in
\cite{CraneD-relational} simplifies somewhat: Give the vertices in the edge
exchangeable hypergraph random labels that are \iid{} and $U(0,1)$
(uniformly distributed on $\oi$), and independent of the edges.
Then the edges become multisets in $\oi\qx$, and their distribution is
clearly exchangeable, so by de Finetti's theorem, the edges are given
by the construction above for some random probability measure $\mu$ on
$\csx$, taking $\cS=\oi$.
\end{remark}

It is obvious from the definition that if $\psi:\cS\to\cS_1$ is 
an injective measurable
map of $\cS$ into another measurable (Borel) space $\cS_1$, then $\mu$ is
mapped to a probability measure $\mu_1$ on $\csx_1$, which defines the same
random hypergraphs $\GGm$ and $\Gm$ as $\mu$.
Hence, the choice of Borel space $\cS$ is not important, and we can always
use \eg{} $\cS=\oi$.
Moreover, we can simplify further.

Define the \emph{intensity} of $\mu$ as the measure on $(\cS,\cF)$
\begin{equation}\label{bmu}
  \bmu(A):=\E|A\cap Y|,
\qquad A\in\cF,
\end{equation}
where $Y$ has distribution $\mu$.
Note that for a singleton set $\set{x}$, 
$|\set{x}\cap Y|=\indic{x\in Y}$, and thus
\eqref{bmu} yields
\begin{equation}\label{bmux}
  \bmu(\set x)=\P(x\in Y).
\end{equation}
We have $\bmu(A)=\sum\nnn\bmu_n(A)$, where 
$\bmu_n(A):=\E\bigpar{|A\cap Y|\cdot\indic{|Y|=n}}$, and since each $\bmu_n$ is a
finite measure, it follows that 
the set of atoms
\begin{equation}\label{atoms}
\cA:=\set{x\in\cS:\bmu(\set {x})>0}
\end{equation}
is a countable (finite or infinite) subset of $\cS$.
By \eqref{bmux} and \eqref{atoms},
if $x\notin\cA$, then $\P(x\in Y)=0$.
Hence, in the construction of $\GGm$, 
if an edge $Y_i$ has a vertex $x\notin \cA$, then
\as{} $x\notin Y_j$ for every $j\neq i$. Consequently, a vertex $x\notin\cA$
of $\GGq_\infty$
\as{} appears in only one edge. 
(Such a vertex is called a \emph{blip} in \cite{CraneD-relational}.)
On the other hand, if $x\in\cA$, so
$\P(x\in Y)=\bmu\pset x>0$, then by the law of large numbers, \as{} $x$
belongs to infinitely many edges $Y_i$ of $\GGq_\infty$.

It follows that when constructing the hypergraphs $\GGm$, 
if the edge $Y_i=\set{y_{i1},\dots,y_{in_i}}$,
we do not have
to keep track of the vertex labels $y_{ij}$ unless they belong to $\cA$;
any $y_{ij}\notin\cA$ will be a blip not contained in any other edge and the
actual value of $y_{ij}$ may be forgotten.
(Except that if we allow repeated vertices in the edges, 
see \refR{Rmulti}, then we still have to
know whether two vertex labels $y_{ij}$ and $y_{ik}$ on the same edge
are the same or not.)

Now, enumerate $\cA$ as $\set{a_i}_{i=1}^N$, where
$N\le\infty$, and replace, for every 
multiset $Y=(y_1,\dots,y_\ell)\in\csx$,  every vertex label $y_{j}=a_k$ for some
$a_k\in\cA$ by the new label $y'_j=k$, and 
the vertex labels $y_{j}\notin\cA$ on $Y$ by $0, -1,
\dots$.
(For definiteness,
we may assume that $\cS\subseteq\oi$ so $\cS$ is ordered, and take the
labels in order in case $Y$ has more than one vertex label not in $\cA$.)
This maps $\mu$ to a probability measure $\mu'$
on the set $\bbZ\qx$ of finite 
multisets of integers, and it follows from the discussion above that we can
recover the random hypergraphs $\GGm$ from $\mu'$ by 
the construction in \refD{D1}, if we
first replace each
vertex label $y_j'\in\set{0,-1,\dots}$ by a random label with a continuous
distribution in some set, for example $U(0,1)$, making independent choices
for each $Y_i$.
Equivalently, and more directly, we obtain $\GGm$ from the probability
measure $\mu'$ on $\bbZ\qx$ by 
the following construction,
which is the original definition by \citet{CraneD-edge,CraneD-relational}.
\begin{definition}[\citet{CraneD-edge,CraneD-relational}] \label{DCD}
  Given a probability measure $\mu$  on $\bbZ\qx$,  
we define a sequence of finite random (multi)hypergraphs $(\GGm)\mmm$
as in \refD{D1} with the modification that in every edge
$Y_i=\set{y_{i1},\dots,y_{i\ell_i}}$ we replace every vertex label $y_{ij}\le0$
(if any)
with a new vertex that is not used for any other edge.
\end{definition}

Since we ignore the vertex labels in $\GGm$, 
it does not matter what labels we use as replacements for $0,-1,\dots$
in \refD{DCD}.
\citet{CraneD-edge,CraneD-relational}
use the same set $0,-1,\dots$ of integers, taking the first label
not already used. An alternative is to take random labels, \eg{} \iid{}
$U(0,1)$ as above.

\begin{remark}\label{R00}
  To be precise, \refD{DCD} is the definition
in \cite{CraneD-relational}.
The definition in \citet{CraneD-edge} treats only the binary case $|Y_n|=2$
in detail; and differs in that only labels $y_i\ge0$ are
used, and that an edge $\set{0,0}$ is replaced by an edge $\set{z_1,z_2}$ with
two new vertex labels $z_1$ and $z_2$. 

This version is essentially equivalent;
apart from a minor notational difference, the only difference is that this
version does not allow for ``loop dust'',
where a positive fraction
of the edges are isolated loops.
Cf.~\refR{R00b}.
\end{remark}

We have shown that \refD{D1} is essentially equivalent to 
 the original definitions by \citet{CraneD-edge,CraneD-relational}.
One advantage of \refD{D1} is that no special treatment of vertex labels
$\le0$ is needed; the blips (if there are any)
come automatically from
the continuous part of the
label distribution; a disadvantage is that this continuous part is arbitrary
and thus does not contain any information.
Another advantage with \refD{D1} is that it allows for arbitrary Borel
spaces $\cS$; even if it usually is convenient to use $\cS=\bbN$ to label
the vertices, it may in some examples be natural to use another set $\cS$.

\begin{remark}\label{Rcai}
The construction in \citet{CampbellCaiBroderick} is stated differently, but is
equivalent. It uses a generalization of Kingman's paintbox construction  
of exchangeable partitions; in the version in \citet{CampbellCaiBroderick},
the paintbox consists of families $(C_{kj})_{k,j\ge1}$ and 
$(C'_{jl})_{j,l\ge1}$ of subsets of $\oi$; it is assumed that every $x\in\oi$
is an element of only finitely many of these sets, and that for each $j$ and
$k\neq l$,
$C_{jk}\cap C_{jl}=\emptyset$ and $C'_{jk}\cap C'_{jl}=\emptyset$.
(In general these sets may be random, but similarly as above, in the
construction we condition on these sets so we may assume that they are
deterministic.) Furthermore, we generate \iid{} $U(0,1)$ random labels
$\phi_k$ and $\phi_{Njl}$ for $k,N,j,l\ge1$. For each $N\ge1$ we construct a
edge $Y_N$ 
by taking a uniformly random point $V_N\in\oi$, independent of
everything else; then, 
for each $(j,k)$ such that $V_N\in C_{jk}$,
$Y_N$ contains
$k$ vertices labelled $\phi_j$,
and
for each
$(j,k)$ such that $V_N\in C'_{jk}$ and  every $l\le k$, $Y_N$ contains
$j$ vertices labelled $\phi_{Njl}$.
(The latter vertices are thus blips.)

Note that this gives the vertices random labels as in \refR{RdeFinetti};
however, we then ignore the vertex labels.
(Actually, in \cite{CampbellCaiBroderick}, each vertex is represented by a
multiset of edge labels (called a \emph{trait}), which contains the label of
each edge that contains the vertex, repeated as many times as the vertex occurs 
in the edge. This is obviously an equivalent way to describe the hypergraph.)

It is obvious that, conditioned on the labels $\phi_k$ and $\phi_{Njl}$,
this construction 
gives a  random multiset with some distribution $\mu$; conversely, every
distribution $\mu$ of a random (finite) multiset can easily be obtained in
this way 
by suitable choices of $C_{jk}$ and $C'_{jk}$. Hence, the construction is
equivalent to the one above. (In our opinion, it is more natural to focus on
the distribution of the edges, since the sets $C_{jk}$ and $C'_{jk}$ in
the paintbox construction have no intrinsic meaning; they are just used to
describe the edge distribution.)
\end{remark}

\subsection{The Poisson version}\label{SSdefPo}

The multihypergraph $\GGm$ has exactly $m$ edges (not necessarily
distinct). It is often convenient to instead consider a Poisson number.
(This was done by \citeauthor{BroderickCai} 
in \cite[Example 2.7]{BroderickCai}.)
It is then natural to consider a continuous-parameter family of hypergraphs,
which we define as follows. We may think of the second coordinate $t$ as time.

\begin{definition}\label{D2}
  Given a probability measure $\mu$  on $\csx$,  
we define a family of random (multi)hypergraphs $(\GGt)_{t\ge0}$ as
follows.
Consider a Poisson point process $\Xi$
on $\csx\times\ooo$ with intensity $\mu\times\dd t$;
then $\Xi$ is a random countably infinite set of points that can be enumerated
 as $\Xi=\set{(Y_i,\tau_i):i\ge1}$
for some $Y_i\in\csx$ and $\tau_i\in\ooo$.
Let, for $0\le t\le\infty$,
the hypergraph $\GG_t$ be the multihypergraph with edges
$E(\GG_t):=\set{Y_i:\tau_i\le t}$, and vertex set $V(\GG_t):=\bigcup_{Y\in
  E(\GG_t)} Y$, \ie, the vertex set spanned by the 
edges. 

 Define $\G_t$ as the simple hypergraph obtained by merging
each set of parallel edges in $\GG_t$ to a single edge, and deleting loops
(together with their incident vertices, unless these also belong to some
non-loop).

Note  that  the random hypergraphs $\GG_t$
and $\G_t$ are 
\as{} finite for every $t<\infty$.
\end{definition}

The projection $\Xi'':=\set{\tau_i}\iii$ of the Poisson process $\Xi$ to the
second coordinate is a Poisson point process on $\ooo$ with intensity 1, and we
may and will
assume that the points of $\Xi$ are enumerated 
with $\tau_i$ in increasing order; thus \as{}
$0<\tau_1<\tau_2<\dots$.
Let $N(t)$ be the number of points of $\Xi$ in $\csx\times[0,t]$, \ie{}
\begin{equation}
N(t):=\bigabs{\Xi\cap(\csx\times[0,t])}=\max\set{i:\tau_i\le t};  
\end{equation}
this is a Poisson
counting process on $\ooo$ and $N(t)\sim\Po(t)$. Conversely, $\tau_m$ is the
time the process $N(t)$ reaches $m$, so the increments $\tau_m-\tau_{m-1}$
(with $\tau_0:=0$) are \iid{} and $\Exp(1)$, and $\tau_m$ has the Gamma
distribution $\Gamma(m)$.
Moreover, the random multisets $Y_i$ are \iid{} with distribution $\mu$ and
independent of $\set{\tau_i}$, so they can be taken as the $Y_i$ in
\refD{D1}, which leads to the following simple relation between the two
definitions. 

\begin{proposition}\label{Pt}
If $\mu$ is a probability measure on $\csx$, then
  the random hypergraphs constructed in Definitions \ref{D1} and \ref{D2}
are related by
$\GGm=\GG_{\tau_m}$ and thus
$\Gm=\G_{\tau_m}$,
and conversely\/
$\GG_t=\GGq_{N(t)}$ and\/ $\G_t=\Gq_{N(t)}$.
\qed
\end{proposition}
Although we usually tacitly consider $t<\infty$, we may here also take
$t=\infty$: $\GGq_\infty=\GG_\infty$ and $\Gq_\infty=\G_\infty$.

Note that the relations in \refP{Pt}
hold not just for a single $m$ or $t$, but also
for the entire processes. Hence, asymptotic results, and in particular \as{}
limit results, are (typically)
easily transfered
from one setting to the other.

\begin{remark}\label{Rcond}
  Instead of stopping at the random time $\tau_m$, we can also obtain
  $\GGm$ and $\Gm$ from $\GG_t$ and $\G_t$ 
by conditioning on $N(t)=m$, for any fixed $t>0$.
\end{remark}

\begin{remark}\label{Rpo}
  One reason that the Poisson version is convenient is that different edges
  appear independently of each other.
If we for convenience assume that there are no blips, we may as explained
above assume that $\cS=\bbN$, so $V(\GG_t)\subseteq\bbN$.
In this case, the number of copies of an edge $I\in\csx$ in 
$\GG_t$ has the
Poisson distribution
$\Po(t\mu\pset{I})$,
and these numbers are independent for different $I\in\csx$.
(In the case $\mu$ is random, this holds conditionally on $\mu$, but not
unconditionally.) 
\end{remark}

\subsection{Unnormalized measures}\label{SSmufinite}

We have so far assumed that $\mu$ is a probability measure.
This is very natural, but we can  make a trivial
extension to arbitrary finite measures. This will not produce any new random
hypergraphs but
it is  convenient; for example, it means that we do not have  to normalize
the measure in the examples in later sections.

When necessary, we denote the measure used in the construction of our random
hypergraphs by a subscript; we may thus write \eg{} $\Gq_{m,\mu}$.

\begin{definition}\label{Dfinite}
  Let $\mu$ be a finite measure on $\csx$,
not identically zero.
Let $\mu_0$ be the probability measure $\mu_0:=\norm{\mu}\qw\mu$, 
and define $\GGq_{m,\mu}:=\GGq_{m,\mu_0}$.
Furthermore, define $\GG_{t,\mu}$ as in \refD{D2}.
Let, as usual, $\Gq_{m,\mu}$ and $\G_{t,\mu}$ be the corresponding simple
graphs. 
\end{definition}

Thus, $\mu=c\mu_0$, 
where $c:=\norm{\mu}=\mu(\csx)$.
It is  obvious that, using obvious notation,
the Poisson process $\Xi_\mu$ can be obtained from $\Xi_{\mu_0}$ by
rescaling the time: If $\Xi_{\mu_0}=\set{(Y_i,\tau^0_i)}$, we can
take
$\Xi_\mu=\set{(Y_i,c\qw \tau^0_i)}$, and thus
$\GGq_{t,\mu}=\GGq_{ct,\mu_0}$.
Hence, the random hypergraph process defined by $\mu$ is the same as for
$\mu_0$, except for a simple deterministic change of time.
This implies the following result.

\begin{proposition}\label{Ptfin}
  \refP{Pt} extends to arbitrary finite measures $\mu$ (not identically
  zero), with
stopping times $\tau_m$ that are the partial sums $\sum_{i=1}^mT_i$
of \iid{} random variables
$T_i\sim\Exp(\norm\mu)$. 
\end{proposition}

In particular, the law of large numbers yields, as \mtoo,
\begin{equation}\label{taulim}
  \tau_m/m\asto \norm\mu\qw.
\end{equation}

\begin{remark}
  \refD{D2} can be employed also when $\mu$ is an infinite, \gsf{} measure.
In this case, $\GG_t$ has \as{} an infinite number of edges for every $t>0$.
We will not consider this case further.
\end{remark}

\section{Random graphs}\label{Sgraphs}

From now on, we consider the graph case, where $\mu$ is a finite measure on 
$\csxx2=\set{\set{x,y}:x,y\in\cS}$. This allows for the presence of loops;
often we consider $\mu$ supported on $\cszz2=\set{\set{x,y}:x\neq y}$,
and then there are no loops.

As explained in \refS{Sdef}, in particular \refD{DCD},
if there are no blips (\ie, if the intensity
$\bmu$ is discrete), we may without loss of generality assume that
$\cS=\bbN$, and if there are blips, we may assume that $\cS=\bbNow$ with the
special convention that 0 and $-1$ are interpreted as blips.
Unless stated otherwise, we use this version, and we then write
$\mu_{ij}$ for $\mu(\set{i,j})$;
we say that $\mu_{ij}$ is the \emph{intensity} of edges $ij$.
Thus,
$(\mu_{ij})$ is an infinite symmetric matrix of non-negative numbers,
with indices in $\bbNow$ (or in $\bbN$ if there are no blips); note that,
because we consider undirected edges, the total mass of $\mu$ is
\begin{equation}
  \norm{\mu}=\frac12\sum_{i,j:\;i\neq j}\mu_{ij}+\sum_i \mu_{ii}.
\end{equation}

We assume that $0<\norm\mu<\infty$, or equivalently that
$\sum_{i,j}\mu_{ij}$ is finite (and non-zero),
but we do not insist on $\mu$ being a probability measure.
As described in \refSS{SSmufinite}, we can always normalize $\mu$ to 
the probability measure $\norm{\mu}\qw\mu$ when desired.

We also define (for $i\ge1$) 
\begin{equation}\label{mui}
  \mu_i:=\sum_j\mu_{ij},
\end{equation}
this is the total intensity of edges adjacent to vertex $i$.

\begin{remark}\label{Rloop}
  The diagonal terms $\mu_{ii}$ correspond to loops.
Loops appear naturally in some examples, see \eg{} \refE{Erank1} below,
but we are often interested in examples without loops, and then take 
$\mu_{ii}=0$. Moreover, in the construction of the simple graphs $\Gm$ and
$\Gt$ we delete loops, so it is convenient to take $\mu_{ii}=0$ and avoid
loops completely.
Note that deleting all loops from $\GGt$ is equivalent to conditioning
$\GGt$ on containing no loops; this is also equivalent to changing every
$\mu_{ii}$ to 0. (For $\Gm$ this is not quite true, since the number of
non-loop edges may change; however, the difference is minor.)

Note also the in the construction leading to \refD{DCD}, in the graph case,
vertex label $-1$ is used only for the edge $\set{0,-1}$, so we may (and
will) assume that $\mu_{i,-1}=0$ unless $i=0$.
\end{remark}

Suppose now that we are given such a matrix $(\mu\xij)_{i,j\ge-1}$.
We can decompose the matrix into the three parts
$(\mu\xij)_{i,j\ge1}$,
$(\mu_{i0})_{i\ge1}$, 
$(\mu_{i0})_{i\in\set{0,-1}}$,
which by the construction and properties of Poisson processes 
correspond to a decomposition of
the Poissonian multigraph $\GG_t$ as a union of 
three parts,
which are
independent random graphs:
{\addtolength{\leftmargini}{-16pt}
\begin{description}
\item [Central part] 
The edges $ij\in\GG_t$ with $i,j\in\bbN$.
\item [Attached stars]For each $i\ge1$ a star with $\Po(t\mu_{i0})$ edges
  centred at $i$. 
\item [Dust]$\Po(t\mu_{00})$ isolated loops and 
$\Po(t\mu_{0,-1})$ isolated edges.
\end{description}}
Moreover, the Poisson random variables above, for different $i$ and for the
two types of dust, are independent. 
The vertex set is by definition the set of
endpoints of the edges, so there are no isolated vertices. The edges and
loops in the 
dust are always isolated, 
\ie{} with endpoints that are blips (have no other edges). 
Similarly, the peripheral vertices in the attached stars are blips without other
edges, while the central vertex $i$ may, or may not, also belong to the
central part.

Note that multiple edges only occur in the central part.

\begin{remark}\label{R00b}
  We have here discussed the model in full generality,
but it is obvious that the main interest is in the central part, and 
all our examples will be with $\mu$ supported on $\bbN\times\bbN$, \ie,
without dust and attached stars. (Of course, there may be other stars or
isolated edges, created in the central part.)

In particular, the dust part is quite trivial, and the dust loops are even less
interesting than the dust edges. In a case with dust but no loops in the dust, 
it is convenient to relabel $\mu_{0,-1}$ as
$\mu_{00}$, so $\mu$ is a symmetric matrix with index set  $\bbNo$;
this corresponds to using the version of the definition in
\cite{CraneD-edge},
see \refR{R00}.
\end{remark}

\subsection{A comparison with vertex exchangeable graphs}
\label{SSvex}
  Consider the case without dust, attached stars and loops, so $\mu$ is
  supported on 
  $\bbN\times\bbN$, with $\mu_{ii}=0$.
Then $\Gt$ has $\Po(t\mu_{ij})$ edges $ij$, for every pair of distinct
integers $i,j\in \bbN$.

We may compare this to the vertex exchangeable random graphs  studied by 
\eg{} \citet{BCLSV1,BCLSV2}, \citet{SJ209}, \citet{Lovasz} 
and their generalizations by \citet{CaronFox},
	\citet{VR}
and
\citet{BCCH16},
see also \citet{OR},
\citet{HSM15},
and \citet{SJ311}.

In the classical case \cite{BCLSV1,BCLSV2,Lovasz},
with a standard graphon $W$ defined on a probability
space $(\gO,\nu)$, the vertex exchangeable
random graph $G(n,W)$ has a given number $n$ of vertices
and is constructed by giving each vertex $i$ a random ``type'' $x_i$
in $\gO$ with distribution $\nu$, independently of all other types.
Then, conditionally on the types, an edge $ij$ is added with probability
$W(x_i,x_j)$, independently for each pair $i,j$ of distinct vertices. 
The generalization to graphons
on $\bbRp$ or another $\gs$-finite measure space $(\gO,\nu)$
\cite{CaronFox,VR,BCCH16} is similar, \cf{} \refS{Sgraphex}: 
then $\gO$ again can be regarded as a
space of types, and the random graph, here denoted $\G(t,W)$, 
has a random (and generally infinite) set of vertices with types that
are given by a Poisson point process on $\gO$ with intensity $t\nu$.
Then, conditionally on the set of vertices and their types, edges are added
as in the classical case. (Finally, we usually delete all isolated vertices, so
that the result is a graph without isolated vertices as in the construction
in \refS{Sdef} above.)
In both cases, a natural multigraph version is to instead add a Poisson
number  $\Po(W(x_i,x_j))$ of copies of the edge $ij$. 
(Cf.\ \eg{} \cite[Remark 2.4]{SJ178}. Note that if the
$W(x_i,x_j)$ are small, then the standard (Bernoulli) and the Poisson
versions are almost the same.)

The Poisson versions of the edge exchangeable and vertex exchangeable random
graphs thus add edges in the same way, if we condition on the types of the
vertices in the latter and let $\mu_{ij}=t\qw W(x_i,x_j)$.
However, the vertices are constructed in very different ways.
To see the similarities and differences clearly, consider the case where the
type space $\gO=\bbN$, with some (finite or infinite) measure $\nu$.
Then the vertex exchangeable
$\G(t,W)$ has a Poisson number $\Po(t\nu\set{i})$ of vertices of type
$i$, for each $i\in\bbN$, 
while the edge exchangeable $\Gt$ has at most 
one vertex $i$ for each $i\in \bbN$.
(We can reformulate the construction of $\Gt$ and say that we start with
exactly one vertex of type $i$ for every $i\in\bbN$,
and then remove all isolated vertices after having added edges.)  

Moreover, although for a fixed $t$, each pair of distinct vertices of types $i$
and $j$ has $\Po(W(i,j))$ edges between them in $\G(t,W)$ and
$\Po(t\mu_{ij})$ edges in $\Gt$, which coincide if $W(i,j)=t\mu_{ij}$, we
see that if we keep $W$ and $\mu$ fixed and increase $t$, the two families
$\G(t,W)$ and $\Gt$ behave differently:
In $\Gt$ the number of edges between each pair of vertices increases
linearly as $t$ increases, the number of vertices increases more slowly
(by \refC{Cot} below; recall that we only keep vertices with at least one
edge),
and there is at most one vertex of each type.
In $\G(t,W)$, the number of vertices of each type increases linearly, while
the number of edges between each pair of vertices remains the same. 

 
\section{Numbers of vertices and edges}\label{Snumber}
By construction, the number of edges is $m$ in the multigraph $\GGm$ and
random $\Po(t\norm{\mu})$ in the multigraph $\GGt$.
The numbers of vertices in the graphs 
and the numbers of edges in the simple graphs $\Gm$
and $\Gt$ are somewhat less immediate, and are studied in this section.

We use the notation of \refS{Sgraphs}, and assume that we are given
a (deterministic) matrix $\mu=(\mu\xij)$ of intensities.
Moreover, for simplicity we assume that 
$\mu$ is concentrated on $\bbN\times \bbN$, so there are no attached stars
and no dust, and that
$\mu_{ii}=0$ for every $i$, so there
are no loops.
We consider briefly the case with dust or attached stars in \refSS{SSdust}.

Note that $\Gm$ 
is a simple graph without isolated vertices,
and thus
\begin{equation}
\label{vev}
  \tfrac12 v(\Gm)\le e(\Gm)\le \binom{v(\Gm)}2\le \tfrac12 v(\Gm)^2.
\end{equation}
Recall that $\Gm$ is \emph{dense} if $e(\Gm)\asymp v(\Gm)^2$,
\emph{sparse} if $e(\Gm)=o(v(\Gm)^2)$,
and \emph{extremely sparse} if $e(\Gm)\asymp v(\Gm)$ as \mtoo.
By Propositions \ref{Pt} and \ref{Ptfin}, these are equivalent to the
corresponding conditions for $\Gt$. 

The number of edges in $\Gm$ is the number of different values
taken by the \iid{} sequence
$Y_1,\dots,Y_m$. Equivalently, it is the number of occupied bins if $m$
balls are thrown independently into an infinite number of boxes, with the
probability $\mu_{ij}$ (normalized if necessary) for box $\set{i,j}$.
Such numbers have been studied in, for example,
\cite{Darling67, Karlin,KestenMR,Dutko,SJ189}, 
where central limit theorems have been proved under various
assumptions, see  \refT{Teve} below.
These results are often proved using Poissonization, which in our setting is
equivalent to considering $\Gt$ instead of $\Gm$.
We too find it convenient to first study the Poisson version.

The Poisson model  is convenient because, as said before, 
edges $ij$ arrive according to a Poisson
process with intensity $\mu_{ij}$ and these
Poisson processes are independent for different pairs \set{i,j}.
Let  $N_{ij}(t)$ be the number of copies of the edge $ij$ in $\GGt$, 
and let $N_i(t)$ be the degree of vertex $i$ in $\GGt$.
Then
\begin{equation}\label{Nij}
  N_{ij}(t)\sim\Po\bigpar{t\mu_{ij}},
\end{equation}
and, recalling \eqref{mui},
\begin{equation}\label{Ni}
  N_i(t)=\sum_{j\neq i} N_{ij}(t)\sim\Po\bigpar{t\mu_i}.
\end{equation}
Moreover,
let
$T_i\sim \Exp(\mu_i)$ and $T\xij\sim \Exp(\mu\xij)$ be the random times that
the first edge at $i$ and the first edge $ij$ appear, respectively.
Thus,
$N_i(t)\ge1\iff T_i\le t$
and $N\xij(t)\ge1\iff T\xij\le t$.

By the construction of $\Gt$,
\begin{align}
  v(\Gt)=v(\GGt) &= \sum_i \indic{N_i(t)\ge1} 
=\sum_i \indic{T_i\le t},\label{vGt}\\
  e(\Gt) &= \sum_{i<j} \indic{N_{ij}(t)\ge1}
=\sum_{i<j}\indic{T\xij\le t}.\label{eGt}
\end{align}

Recall that for every fixed $t$, the numbers $N_{ij}(t)$ are independent
random variables, and thus the indicators 
in the sums \eqref{eGt} are independent.
However, the numbers $N_i(t)$ and the indicators in the sums in \eqref{vGt}
are dependent, 
which is a complication. (For example, $v(\Gt)=1$ is impossible, since there
are no isolated vertices and no loops.)

We give first a simple lemma for the type of sums in \eqref{eGt}, where the
terms are independent.

\begin{lemma}\label{LW}
  Let $Z_i\sim\Exp(\gl_i)$, $i=1,2,\dots$, be independent exponential random
  variables with $\gl_i\ge0$ and
$0<\sumi \gl_i<\infty$,
and let $W(t):=\sumi\indic{Z_i\le t}$.
Then, the following hold.
\begin{romenumerate}
\item \label{LWa}
For every $t\ge0$,
\begin{equation}\label{EW}
  \E W(t) = \sumi\P\bigpar{Z_i\le t}
=\sumi \bigpar{1-e^{-\gl_i t}}<\infty
\end{equation}
and thus \as{} $W(t)<\infty$ for every $t\ge0$.
Furthermore,
$\E W(t)$ is a 
 strictly increasing and concave
continuous function of $t\ge0$ with $\E W(0)=0$
and $\E W(t)/t\to0$ as \ttoo.
\item \label{LWasymp}
For $t>0$,
\begin{equation}\label{lwasymp}
  \E W(t)\asymp \sumi \bigpar{1\land(\gl_i t)}.
\end{equation}

\item \label{LWvar}
For every $t\ge0$,
\begin{equation}\label{lwvar}
  \Var\bigpar{W(t)}
=\sumi e^{-\gl_i t} \bigpar{1-e^{-\gl_i t}}
\le\E W(t).
\end{equation}

\item \label{LWlim}
Let $L:=|\set{i:\gl_i>0}|\le\infty$.
Then
as \ttoo, $\E W(t)\to L$,
$W(t)\asto L$ and 
\begin{equation}\label{lwlim}
  \frac{W(t)}{\E W(t)}\asto 1.
\end{equation}
\item \label{LW'}
If $(t_n)$ and $(t'_n)$ are two sequences of positive numbers with
$t_n'/t_n\to1$, then $\E W(t_n')/\E W(t_n)\to1$.
\end{romenumerate}
\end{lemma}

\begin{proof}
This is presumably all known, 
but it seems easier to give a proof than to find
references.
Note that $W(t)$ is increasing as a function of $t$.

\pfitemref{LWa}
The calculation \eqref{EW} of the expectation is immediate, and the sum is
finite because
$1-e^{-\gl_i t}\le \gl_i t$.
Hence $W(t)$ is \as{} finite for, say, each integer $t$, and thus for all
$t\ge0$.

It follows by \eqref{EW} that $\E W(t)$ is strictly increasing and concave.
Moreover, the sum converges uniformly on every finite interval $[0,T]$,
and thus $\E W(t)$ is continuous.
Finally, $\E W(t)/t=\sumi(1-e^{-\gl_i t})/t$, where each summand tends to 0
as \ttoo, and is bounded by $\gl_i$. Hence $\E W(t)/t\to0$ as \ttoo{} by
dominated convergence of the sum.

\pfitemref{LWasymp}
An immediate consequence of \eqref{EW} and $1-e^{-x}\asymp 1\land x$.

\pfitemref{LWvar}
Since the summands in $W$ are independent,
\begin{equation}
  \begin{split}
      \Var\bigpar{W(t)}
&=\sumi \P(Z_i\le t)\bigpar{1-\P(Z_i\le t)}
=\sumi e^{-\gl_i t} \bigpar{1-e^{-\gl_i t}}
\\
&\le \sumi \P(Z_i\le t)=\E W(t).
  \end{split}
\end{equation}

\pfitemref{LWlim}
First, by \eqref{EW} and monotone convergence, as \ttoo,
\begin{equation}\label{ewl}
  \E W(t)\to \sumi\P(Z_i<\infty)=\sumi\indic{\gl_i>0}=L.
\end{equation}
Furthermore, if $L<\infty$, then \as{} $W(t)= L$ for all large $t$, and thus 
 \eqref{lwlim} holds.

Suppose now that $L=\infty$.
Then $\E W(t)\to\infty$ by \eqref{ewl}.
Let $\gd\in(0,1)$, let $a:=1+\gd$ and choose, for $n\ge1$,
$t_n>0$ such that $\E W(t_n)=a^n$.
(This is possible by \ref{LWa}.)

By \eqref{lwvar} and Chebyshev's inequality, for any $t>0$,
\begin{equation}
  \P\Bigpar{\Bigabs{\frac{W(t)}{\E W(t)}-1}>\gd} 
\le \frac{\Var(W(t))}{(\gd \E W(t))^2}
\le \frac{1}{\gd^2 \E W(t)}.
\end{equation}
Hence, by our choice of $t_n$ and the Borel--Cantelli lemma,
\as{} there exists a (random) $n_0$ such that
$1-\gd \le W(t_n)/\E W(t_n)\le 1+\gd$ for $n\ge n_0$.
This, and the fact that $W(t)$ is increasing, implies that if $t\ge t_{n_0}$,
and we choose $n\ge n_0$ such that $t_n\le t<t_{n+1}$, then
\begin{equation}
  W(t)\le W(t_{n+1})\le (1+\gd) a^{n+1}
=(1+\gd)^2 \E W(t_n)
\le (1+\gd)^2 \E W(t),
\end{equation}
and similarly
\begin{equation}
  W(t)\ge W(t_{n})\ge (1-\gd) a^{n}
\ge (1-\gd)^2 \E W(t_{n+1})
\ge (1-\gd)^2 \E W(t).
\end{equation}
Consequently, \as{}
\begin{equation}
  (1-\gd)^2\le\liminf_{\ttoo} \frac{W(t)}{\E W(t)}
\le\limsup_{\ttoo} \frac{W(t)}{\E W(t)}
\le (1+\gd)^2.
\end{equation}
Since $\gd$ is arbitrarily small, \eqref{lwlim} follows.

\pfitemref{LW'}
By \ref{LWa},
$\E W(t)$ is increasing, and furthermore it is concave with $\E W(0)=0$, and
thus 
$\E W(t)/t$ is decreasing on $(0,\infty)$. Hence,
\begin{equation}
  \min\set{1,t_n'/t_n}
\le \E W(t_n')/\E W(t_n)
\le \max\set{1,t_n'/t_n}
\end{equation}
and the result follows.
\end{proof}

In order to extend this to the dependent sum \eqref{vGt}, we use a lemma.

\begin{lemma}\label{Lx2}
  Let $(I\xij)_{i,j=1}^N$ be a finite or infinite symmetric array of random
  indicator variables, with $\set{I\xij}_{i\le j}$ independent.
Let $I_i:=\max_j I\xij$, and $W:=\sum_i I_i$.
Then
\begin{equation}
  \label{lx2}
  \Var W \le 2 \E W.
\end{equation}
\end{lemma}

\begin{proof}
Assume first that $N<\infty$.
Let $\bI\xij:=1- I\xij$ and $\bI_i:=1-I_i=\prod_j\bI\xij$.
Let $q\xij:=\E\bI\xij=1-\E I\xij$.

Fix $i$ and $j$ with $i\neq j$, and let $\bIx_i:=\prod_{k\neq j} \bI_{ik}$ and
$\bIx_j:=\prod_{k\neq i} \bI_{jk}$.
Then $\bI_i=\bI\xij\bIx_i$ and $\bI_j=\bI\xij\bIx_j$, with $\bI\xij$,
$\bIx_i$ and $\bIx_j$ independent, and thus
\begin{equation}
  \begin{split}   
\Cov(I_i,I_j)
&
=\Cov(\bI_i,\bI_j)
=
\E\bigpar{\bI\xij\bIx_i\bIx_j}
-\E\bigpar{\bI\xij\bIx_i}\E\bigpar{\bI\xij\bIx_j}
\\&
= q\xij \E\bIx_i\bIx_j - \bigpar{q\xij \E\bIx_i}\bigpar{q\xij\bIx_j}
= q\xij (1-q\xij)\E\bIx_i\bIx_j .
  \end{split}
\end{equation}
In particular,
\begin{equation}\label{hc}
\Cov(I_i,I_j)\le (1-q\xij)\E\bIx_i
=\P\bigpar{I\xij=1, \, I_{ik}=0 \text{ for } k\neq j}.
\end{equation}
Summing over $j$, we obtain for every $i$,
since the events $\cE_j:=\set{I\xij=1, \,
I_{ik}=0 \text{ for } k\neq j}$ in \eqref{hc}
are disjoint and with union $\set{\sum_jI\xij=1}=\set{I_i=1}$,
\begin{equation*}
  \begin{split}
    \sum_{j\neq i}\Cov(I_i,I_j)
\le \sum_{j\neq i}(1-q\xij)\E\bIx_i
=\P\Bigpar{\sum_j I\xij=1}
\le\P(I_i=1)
=\E I_i.
  \end{split}
\end{equation*}
Furthermore,
$\Cov(I_i,I_i)=\E I_i - (\E I_i)^2\le \E I_i$.
Consequently, for every $i$,
\begin{equation}
  \begin{split}
    \sum_{j=1}^n\Cov(I_i,I_j)
\le 2 \E I_i,
  \end{split}
\end{equation}
and \eqref{lx2} follows by summing over $i$.
\end{proof}

\begin{lemma}
  \label{LW2}
Let $(Z\xij)\xij$ be a symmetric array of exponential random
variables with 
$\set{Z\xij}_{i\le j}$ independent and
$Z\xij\sim\Exp(\gl\xij)$,
where
$\gl\xij\ge0$ and
$0<\sum\xij\gl\xij<\infty$.
Let $Z_i:=\inf_j Z\xij$ and $W(t):=\sum_i \indic{Z_i\le t}$.
Then $Z_i\sim\Exp(\gl_i)$ with $\gl_i:=\sum_j\gl\xij$.
Moreover,
all results of \refL{LW} hold except \ref{LWvar}, which is replaced by
\begin{equation}
  \label{lw2var}
\Var\bigpar{W(t)}\le 2\E W(t).
\end{equation}
\end{lemma}
\begin{proof}
It is well-known and elementary that   $Z_i\sim\Exp(\gl_i)$, since
$\xpar{Z\xij}_j$ are independent for every $i$.
Parts \ref{LWa}, \ref{LWasymp} and \ref{LW'} of \refL{LW}
deal only with the expectation,
and their proofs do not need $Z_i$ to be independent.

\refL{Lx2} yields \eqref{lw2var}.

Finally, the proof of \ref{LWlim} holds as before, now using \eqref{lw2var}.
\end{proof}

We return to the random graphs.
We define, for a given measure (matrix) $\mu$,
using Lemmas \ref{LW} and \ref{LW2} together with
\eqref{Nij}--\eqref{eGt},
the functions
\begin{align}
  v(t)&=v(t;\mu):=\E v(\Gt)
=\sumi\bigpar{1-e^{-\mu_i t}}\asymp\sumi\bigpar{1\land(\mu_it)},
\label{vt}\\
  e(t)&=e(t;\mu):=\E e(\Gt)=\sum_{i<j}\bigpar{1-e^{-\mu\xij t}}
\asymp\sum_{i\neq j}\bigpar{1\land(\mu\xij t)}.\label{et}
\end{align}

Since $\Gt$ has no isolated vertices, $e(\Gt)\ge \frac12v(\Gt)$, and thus,
\cf{} \eqref{vev}, 
\begin{equation}\label{e>v}
  e(t)\ge \tfrac12 v(t).
\end{equation}

\begin{theorem}\label{Too}
Assume that $\mu=(\mu\xij)_{i,j=1}^\infty$ is a symmetric non-negative
matrix with  $\mu_{ii}=0$
and $0<\norm{\mu}:=\sum_{i<j}\mu\xij<\infty$.
\begin{romenumerate}
\item \label{Toot}
As \ttoo,
\begin{align}
  v(\Gt)/v(t)&\asto1,\label{toov}\\
  e(\Gt)/e(t)&\asto1.\label{tooe}
\end{align}
Moreover, if $\mu_{ij}>0$ for infinitely many pairs $(i,j)$, then
as \ttoo,
$ v(t)\to\infty$,
$ e(t)\to\infty$ and
$v(\Gt),e(\Gt)\asto\infty$.
\item \label{Toon}
As \mtoo,
\begin{align}
  v(\Gm)/v\bigpar{\norm\mu\qw m}&\asto1,\label{moov}\\
  e(\Gm)/e\bigpar{\norm\mu\qw m}&\asto1.\label{mooe}
\end{align}
In particular, a.s.
\begin{align}
  v(\Gm)&\asymp\sumi\bigpar{1\land(\mu_im)},
\label{vm}\\
  e(\Gm)&
\asymp\sum_{i,j}\bigpar{1\land(\mu\xij m)}.\label{em}
\end{align}
Consequently, if $\mu_{ij}>0$ for infinitely many pairs $(i,j)$, then
as \mtoo,
\as{}
$v(\Gm),e(\Gm)\to\infty$.
\end{romenumerate}
\end{theorem}

\begin{proof}
\pfitemref{Toot}
This is an immediate consequence of 
  \refL{LW}\ref{LWlim} and Lemma \ref{LW2}. 

\pfitemref{Toon} 
Part \ref{Toot} and Propositions \ref{Pt} and \ref{Ptfin}
show that
$v(\Gm)/v(\tau_m)=v(\G_{\tau_m})/v(\tau_m)\asto1$.
Furthermore, $\tau_m\sim \norm\mu\qw m$ by \eqref{taulim}, and thus
$v(\tau_m)\sim v\bigpar{\norm\mu\qw m}$ by 
Lemmas \ref{LW}\ref{LW'} and \ref{LW2}.
Hence \eqref{moov} follows. The proof of \eqref{mooe} is the same.

Finally \eqref{vm}--\eqref{em} follow by \eqref{vt}--\eqref{et},
and the final sentence follows by monotone convergence (or by
\refL{LW}\ref{LWlim}).
\end{proof}

Hence, to find asymptotics of the numbers of vertices and edges in our
random graphs, it suffices to study the expectations in
\eqref{vt}--\eqref{et}. 
In particular, 
we note the following consequences.

\begin{corollary}\label{Ctk}
Assume that $\mu=(\mu\xij)_{i,j=1}^\infty$ is a symmetric non-negative
matrix with  $\mu_{ii}=0$
and $0<\norm{\mu}:=\sum_{i<j}\mu\xij<\infty$.
Then:
\begin{romenumerate}
\item   
 $\Gm$ is 
\as{} dense if and only if $e(t)\asymp v(t)^2$ as \ttoo.
\item   
 $\Gm$ is 
 \as{} sparse if and only if $e(t)=o(v(t)^2)$ as \ttoo.
\item   
 $\Gm$ is 
 \as{} extremely sparse if and only if $e(t)\asymp v(t)$ as \ttoo.
\end{romenumerate}
\end{corollary}
\begin{proof}
  By \refT{Too}\ref{Toon}.
\end{proof}

\begin{corollary}\label{Cot}
Assume that $\mu=(\mu\xij)_{i,j=1}^\infty$ is a symmetric non-negative
matrix with  $\mu_{ii}=0$
and $0<\norm{\mu}:=\sum_{i<j}\mu\xij<\infty$.
Then, \as,
\begin{romenumerate}
\item 
$v(\Gm)=o(m)$ and $e(\Gm)=o(m)$  as \mtoo;
\item 
  $v(\Gt)=o(t)$ and $e(\Gt)=o(t)$  as \ttoo.
\end{romenumerate}
\end{corollary}
\begin{proof}
By \refT{Too}, since $e(t)/t\to0$ and $v(t)/t\to0$ as \ttoo{} by
\refL{LW}\ref{LWa} and \refL{LW2}.  
\end{proof}

\begin{remark}\label{Rmany}
  If we consider the random multigraph $\GGm$ we have (at least in the
  loop-less case, and in general with a minor modification)
  $v(\GGm)=v(\Gm)=o(m)$ by \refC{Cot}, while by definition there are $m$
  edges. Hence, the average degree $2e(\GGm)/v(\GGm)\to\infty$ \as{} as
  \mtoo.
Similarly, the average number of copies of each edge $e(\GGm)/e(\Gm)\to\infty$
\as
\end{remark}

For future use we note also that since $v(t)$ is concave with $v(0)=0$,
for any $C\ge1$,
\begin{equation}\label{vC}
v(t)\le v(Ct)\le Cv(t).
\end{equation}
Hence $v(Ct)\asymp v(t)$ for any constant $C>0$.

We have so far considered only simple first order properties of $v(\Gm)$ and
$e(\Gm)$. For the number of edges, much more follows from the central limit
results in the references mentioned above. In particular, the local and
global central limit theorems in \cite{SJ189} apply and yield
the following.
\begin{theorem}\label{Teve}
Let $\mu$ be as in \refT{Too}.
The following hold
with $O(1)$ bounded by an absolute constant $C$ uniformly for all $n\ge1$,
$x\in\bbR$, and matrices $\mu$.

Let\/ $\gss_m:=\Var(e(\Gm))$. Then
\begin{align}
\P\bigpar{e(\Gm)=\floor{\E e(\Gm)+x\gs_m}}
=\frac{e^{-x^2/2}}{\sqrt{2\pi}\gs_m}
+  \frac{O(1)}{\gss_m}. 
\end{align}
Moreover, assuming for simplicity $\norm\mu=1$,
\begin{align}
\E e(\Gm)&=\E e(\G_m)+O(1)  ,\label{eva}
\\
\Var (e(\Gm))&=\Var( e(\G_m))+O(1)  ,\label{evb}
\end{align}
 and,
recalling \eqref{et} 
and defining
$\tgss_t:=\Var(e(\Gt))$, 
\begin{align}
\P\bigpar{e(\Gm)=\floor{e(m)+x\tgs_m}}
=\frac{e^{-x^2/2}}{\sqrt{2\pi}\tgs_m}
+ \frac{O(1)}{\tgss_m}. 
\end{align}

In particular, if $\mtoo$ and $\tgss_m\to\infty$, then
$e(G_m)-\E e(G_m))/\gs_m\dto N(0,1)$ and
$e(G_m)-e(m))/\tgs_m\dto N(0,1)$.

The $O(1)$ in \eqref{eva}--\eqref{evb} can be replaced by $o(1)$ as $\mtoo$
for a fixed $\mu$.
\end{theorem}

\begin{proof}
  By 
\cite[Theorems 2.1, 2.3, 2.4 and Corollary 2.5, together with Section 9]{SJ189}.
\end{proof}

Note that $e(m)=\E e(\G_m)$ and $\tgss_m=\Var e(\G_m)$ are given by
\eqref{et} and \eqref{lwvar}; they are usually simpler and more convenient to
handle than $\E e(\Gm)$ and $\gss_m=\Var(e(\Gm))$.

We conjecture that similar results holds for $v(\Gm)$, the number of
vertices.
However, we cannot obtain this directly from results on the occupancy
problem
in the same way as \refT{Teve},
again because the variables $N_i(t)$ are dependent.
(The number of vertices corresponds to an occupancy problem where balls are
thrown in pairs, with a dependency inside each pair.)

\begin{problem}
  Show asymptotic normality for $v(\Gm)$ when $\Var(v(\Gm))\to\infty$.
\end{problem}

\subsection{The case with dust or attached stars}\label{SSdust}
We consider briefly the case when
the model contains dust (other than loops) or attached stars.
In this case,  the results are quite different. 
We may for simplicity assume that there are no loops at all, since loops are
deleted in any case.
Thus $\mu_{ii}=0$ for $i\ge0$ and
$\mu_{0i}>0$ for some $i\in\bbNw$.

The number of edges in
the dust and attached stars of $\Gt$ is $\Po(ct)$ with
$c:=\sum_{i=-1}^\infty\mu_{0i}>0$, and thus this number is \as{} $\sim
ct\asymp t$ as \ttoo, by the law of large numbers for the Poisson process. 
(Recall that all edges in the dust and attached stars of
$\GGt$ are simple, so the number of them is the same in $\Gt$ and in $\GGt$.)
It follows by \refP{Pt}
that the number of edges in the dust and attached stars of
$\Gm$ \as{} is $\asymp m$.
Moreover, since each edge in the dust or an attached star 
has at least one endpoint
that is not shared by any other edge, the same estimates
hold for the number of
vertices in the dust and attached stars.
This leads to the following theorem, which shows that if there is any dust
or attached star all, then those parts will dominate the random graphs.

\begin{theorem}\label{Tdust}
  Assume that $\mu_{0i}>0$ for some $i\in\bbNw$.
Then, \as,
\begin{romenumerate}
\item \label{Tdust1}
$v(\Gm)\asymp m$ and $e(\Gm)\asymp m$  as \mtoo;
\item \label{Tdust2}
  $v(\Gt)\asymp t$ and $e(\Gt)\asymp t$  as \ttoo.
\end{romenumerate}
Moreover, \as, all but a fraction $o(1)$ of the edges and vertices are in
the dust or attached stars.

Consequently, the random graphs $\Gm$ are \as{} extremely sparse, but in a
rather trivial way.
\end{theorem}
\begin{proof}
The argument before the theorem shows \ref{Tdust1} and \ref{Tdust2}.

Moreover, \refC{Cot} applies to the central part of $\Gt$
and shows that the number of
edges and vertices there \as{} are $o(t)$, and thus only
a fraction $o(1)$ of all edges and vertices.
By \refP{Pt}, the same holds for $\Gm$.
\end{proof}

\section{Rank 1 multigraphs}\label{Srank1}

We turn to considering specific examples of the construction.
One interesting class of examples are constructed as follows.

\begin{example}[Rank 1]\label{Erank1}
  Let $(q_i)_1^\infty$ be a probability distribution on $\bbN$, and
  construct a sequence of \iid{} edges $e_1,e_2,\dots$, each obtained by
  selecting the two   endpoints as independent random vertices with the
  distribution $(q_i)_i$. (Thus loops are possible.)
Define the
random multigraph $\GGm$
by taking 
the $m$ edges $e_1,\dots,e_m$,
letting the vertex set be the set of their endpoints.
(Equivalently: start with the vertex set $\bbN$ and then remove all isolated
vertices.) 

In other words, let $V_1,V_2,\dots$ be an \iid{} sequence of vertices with
the distribution $(q_i)_i$, and let the edges of $\GGm$ be $V_1V_2, V_3V_4,
 \dots,V_{2m-1}V_{2m}$.

This is clearly a random multigraph of the type constructed in
\refS{Sgraphs},
with
\begin{equation}\label{muij1}
  \mu_{ij}=
  \begin{cases}
	2q_iq_j,&i\neq j,
\\
q_i^2, &i=j.
  \end{cases}
\end{equation}
We thus have, by \eqref{mui},
\begin{equation}\label{mui1}
  \mu_i=\sum_{j\neq i} 2q_iq_j+q_i^2=2q_i-q_i^2.
\end{equation}
In particular, $\mu_i\asymp q_i$.

The corresponding Poisson model $\GGt$ is by \refP{Pt} obtained by taking a
Poisson number of edges $e_1,\dots,e_{N(t)}$, with $N(t)\sim\Po(t)$.

As usual, we obtain the corresponding simple graphs by omitting all
repeated edges and deleting all loops.
\end{example}

We call a random multigraph constructed as in \refE{Erank1}, or equivalently
by \eqref{muij1}, for some (possibly random) probability distribution
$(q_i)_1^\infty$, a \emph{rank 1 edge exchangeable multigraph}, 
for the reason that the matrix
\eqref{muij1} is a rank 1 matrix except for the diagonal entries. 

\begin{remark}
\label{Rloop1}  
The diagonal entries, creating loops, are less important to us.
In the multigraph examples below, it is natural, and simplifies the results, 
to allow loops. However, when we consider the simple graphs $\Gt$ and
$\Gm$, we ignore loops and, see \refR{Rloop}, it is then simpler 
to  modify \eqref{muij1} by taking $\mu_{ii}=0$; 
we still say that the resulting random graphs are rank 1.
\end{remark}

\begin{remark}\label{Rrank1vertex}
Note that the rank 1 random graphs in \cite{SJ178} are 
different; they are simple graphs, and they are
vertex exchangeable or modifications of vertex
exchangeable random graphs, \cf{} \refSS{SSvex}.  
Nevertheless, both types of ``rank 1'' random graphs can be seen as based on
the same idea: each vertex is given an ``activity'' ($q_i$ in our case), and
the probability of an edge between two vertices is proportional to the
product of their activities.
(See the references in \cite{SJ178} for various versions of this idea.)
\end{remark}

Recall that the \emph{configuration model} is an important model for
constructing random multigraphs with a given degree sequence, which is
defined as follows, see \eg{} \citet{Bollobas}.

\begin{definition}[Configuration model]
  Given a sequence $(d_i)_{i=1}^n$ of non-negative integers with $\sum_id_i$
  even,
the random multigraph \ggndd{} is defined by considering a set of $\sum_i d_i$
\emph{half-edges} (or \emph{stubs}), of which $d_i$ are labelled $i$ for
each $i\in[n]$, and taking a uniformly random matching of 
the half-edges; each pair of half-edges is interpreted as an edge
between the corresponding vertices.
\end{definition}

By construction, the multigraph \ggndd{} has degree sequence $\ddn$.
(With a loop counted as 2 edges at its only endpoint.)
Note that the distribution of $\ggndd$ is not uniform over all multigraphs
with this 
degree sequence. (As is well-known, and easy to see, the probability
distribution has a factor (weight) $1/2$ for each loop and $1/\ell!$
for each edge of multiplicity $\ell>1$; in particular, conditioned on being
a simple graph, $\ggndd$ has a uniform distribution.)
Nevertheless, $\ggndd$ has the right distribution for our purposes.

\begin{theorem}\label{Tconfig}
  The random multigraph $\GGm$ constructed in \refE{Erank1}
has, conditioned on its degree sequence  $\ddn$, the same distribution as 
the random multigraph $\ggndd$ constructed by the configuration model for
that degree sequence.

The same holds for $\GGt$.
\end{theorem}

\begin{proof}
  In the construction of $\GGm$ above, 
the sequence $V_1,\dots,V_{2m}$ is \iid, and
  thus exchangeable; hence its distribution is unchanged if we replace 
each $V_i$ by $V_{\pi(i)}$ for a uniformly random permutation $\pi$ of
$[2m]$,
independent of everything else.
Consequently, the distribution of $\GGm$ is the same if we modify the
definition above and let the edges be
$V_{\pi(1)}V_{\pi(2)},\dots,V_{\pi(2m-1)}V_{\pi(2m)}$; but this is the same
as saying that the edges are obtained by taking a random matching of the
multiset \set{V_1,\dots,V_{2m}}, which is precisely what the configuration
model does.
(Note that
the vertex degree $d_i$ is the number of times $i$ appears in 
$V_1,\dots,V_{2m}$.)

The result for $\GGt$ follows, since the degree sequence tells how many
edges there are, so conditioning on the degree sequence implies
conditioning on $e(\GGt)=N(t)$, which reduces to the case of $\GGm$ just
proved, see \refR{Rcond}.
\end{proof}

\begin{remark}\label{Rsufficient}
  In statistical language, the theorem implies that the degree distribution
  is a sufficient statistic for the family of distributions of multigraphs
  $\GGm$ (or $\GGt$) given by \refE{Erank1} with different distributions
  $(q_i)_1^\infty$.
\end{remark}

\begin{example}\label{EER}
  A trivial example of the construction in \refE{Erank1} is obtained by
  fixing $n\ge1$ and letting $q_i=1/n$, $1\le i\le n$, \ie, the uniform
  distribution on $[n]$.
This means that we consider a sequence of \iid{} edges, each obtained by
taking the two endpoints uniformly at random, and independently, from $[n]$.
In other words, the endpoints of the edges are obtained by drawing with
replacement from $[n]$.
This gives the random multigraph process studied in \eg{} \cite{SJ97},
which is a natural multigraph version of the (simple) random graph process
studied by \citet{ER1960}.
\end{example}

The rank 1 random multigraphs in \refE{Erank1} appear also hidden in some
other examples. 

\begin{example}[The Hollywood model]\label{Ehollywood}
The Hollywood model of a random hypergraph
was defined in \citet{CraneD-edge}
using the language of actors participating in the same movie,
see \cite{CraneD-edge} for details.
We repeat their definition in somewhat different words.

The model can be defined by starting with the two-parameter version of the 
Chinese restaurant process, see \eg{} \cite[Section 3.2]{Pitman} 
and \cite{Crane-Ewens}, 
which
starts with a single table with one customer. New customers arrive, one by
one; if a new customer arrives when there are $n$ customers seated at $k$
tables,
with $n_i\ge1$ customers at table $i$, then the new customer is placed: 
\begin{equation}\label{chinese}
\left\{
  \begin{alignedat}{2}
	&\text{at  table $i$ ($1\le i\le k$) with probability}\quad
&& 
\xpqfrac{n_i-\ga}{n+\gth},
\\
&\text{at a new table $k+1$ with probability}&&
	  \xpqfrac{\gth+k\ga}{n+\gth}.
  \end{alignedat}
\right.
\end{equation}
Here $\ga$ and $\gth$ are parameters, and either
\begin{romenumerate}
\item\label{ch1} $0\le\ga\le1$ and $\gth>-\ga$, or
\item\label{ch2} $\ga<0$ and $\gth=N|\ga|>0$ for some $N\in \bbN$.
\end{romenumerate}
In case \ref{ch2}, there are never more than $N$ tables; in case \ref{ch1},
the number of tables grows \as{} to $\infty$.

In the construction of the Hollywood model hypergraph, 
the vertices are the tables in the Chinese restaurant process.
We furthermore draw the sizes of the edges as \iid{} random
variables $\KK_j$ with some distribution $\nu$ on the non-negative integers
$\bbN$. 
The first edge is then defined by (the set of tables of) 
the first $\KK_1$ customers,
the second edge by the next $\KK_2$ customers, and so on.
The random hypergraph $\G_m$ with $m$ edges is thus described by the
first $\KK_1+\dots+\KK_m$ customers. 

 A standard calculation shows that the sequence of table numbers is
exchangeable, except that the numbers occur for the first time in the
natural order; to be precise, the probability of any finite sequence of
table numbers, such that the first 1 appears before the first 2, and so on,
depends only on the number of occurences of each number. 
Consequently, as noted in \cite{CraneD-edge},
since we ignore vertex labels, 
and the sequence $\KK_1,\KK_2,\dots$ is \iid{} and independent of the Chinese
restaurant process,
the random hypergraph
$\GG_\infty$
is exchangeable, and by the representation theorem by
\citet{CraneD-edge,CraneD-relational}, see \refR{RdeFinetti}, the Hollywood
model can be constructed as in \refD{D1} for some random measure $\mu$ on
$\bbN$.

We can see this more concretely by
replacing the table labels $i\in \bbN$ by \iid{} random labels $U_i\sim U(0,1)$;
then the sequence of table labels of the customers is exchangeable.
Hence, by de Finetti's theorem, there exists a random probability measure
$\hP$ on $\oi$ such that conditioned on $\hP$,
the sequence of (new) table labels is an \iid{} sequence with distribution
$\hP$. 
Clearly, the random measure $\hP=\sum_i \tP_i\gd_{U_i}$ for some random
sequence $\tP_i$ of numbers with $\sum_i\tP_i=1$.
Furthermore, by the law of large numbers, for every $i\in\bbN$, 
$\tP_i$ equals \as{} the asymptotic frequency of customers
sitting at the
table originally labelled $i$ in the Chinese restaurant process.
Hence, the random probability
measure $\tP=(\tP_i)_1^\infty$ on $\bbN$ has the distribution
$\GEM(\ga,\gth)$, see \cite[Theorem 3.2 and Definition 3.3]{Pitman}.
Consequently, the Hollywood model hypergraph can be constructed as follows: 
Let the random probability measure $\tP$ on $\bbN$ have the distribution
$\GEM(\ga,\gth)$; conditionally
given $\tP$ take an infinite \iid{} sequence of vertices with
distribution $\tP$; construct the edges by taking the first $X_1$ vertices,
the next $X_2$ vertices, \dots; finally, ignore the vertex labels.

We specialize to the graph case and assume from now on that $\KK_j=2$
(deterministically).
Thus edges are constructed by taking the customers pairwise as they arrive.
We then see by comparing the constructions above and in \refE{Erank1}
that
the Hollywood model yields the same result as the rank 1 model in \refE{Erank1},
based on a random probability distribution with distribution
$\GEM(\ga,\gth)$.

Since the order of the probabilities $q_i$ does not matter in \refE{Erank1}, 
we obtain the same result if we reorder the probabilities $\tP_i$ in
decreasing order; this gives the Poisson--Dirichlet distribution
$\PD(\ga,\gth)$ \cite[Definition 3.3]{Pitman}, 
and thus the Hollywood model is also given by the rank 1
model based on $\PD(\ga,\gth)$. 

\refT{Tconfig} shows that an yet another way to define the Hollywood model
multigraph $\GGm$ is to take the configuration model  where the degree
sequence $(d_i)_1^m$ is the (random) sequence of numbers of customers at
each table in the Chinese restaurant process when there are $2m$ customers.
\end{example}

\begin{example}
 \citet{Pittel} considers the random multigraph process with a fixed vertex
set $[N]$, where edges are added one by one (starting with no edges)
such that 
the probability that a new edge joins two distinct vertices $i$ and $j$
is proportional to $2(d_i+\ga)(d_j+\ga)$, and the probabiity that the new
edge is a loop at $i$ is proportional to $(d_i+\ga)(d_i+1+\ga)$; 
here $d_i$ is the current
degree of vertex $i$ and $\ga>0$ is a fixed parameter.
(\cite{Pittel} considers also the corresponding process for simple graphs;
we do not consider that process here.)

It is easily seen that this 
multigraph process can be obtained as above, with a minor modification of
the Chinese restaurant process. Consider now a restaurant with a fixed
number $N$ of tables, initially empty, 
and seat each new customer at table $i$ with probability
\begin{equation}\label{japanese}
(n_i+\ga)/(n+N\ga),   
\end{equation}
where  $n_i\ge0$ is the number of customers at
table $i$ and $n$ is their total number. Then construct edges by taking the
customers pairwise, as above; this yields the multigraph process just described.

Furthermore,
although this construction uses a modification of the Chinese restaurant
process, we can relabel the tables in the random order that they are
occupied.
It is then easily seen that we obtain the Chinese restaurant process 
\eqref{chinese} with parameters $(-\ga,N\ga)$. 
Since the vertex labels are ignored, this means that
Pittel's multigraph process is the same as the Hollywood model with
parameters $(-\ga,N\ga)$.
Consequently, it can be defined by the rank 1 model in \refE{Erank1} with
the random probability distribution $\GEM(-\ga,N\ga)$ on $[N]\subset\bbN$, or,
equivalently, the random probability distribution $\PD(-\ga,N\ga)$.

Moreover, the restaurant process \eqref{japanese} can be seen as a \Polya{}
urn process, with balls of $N$ different colours and initially $\ga$ balls
in each colour, where $n_i$ is
the number of additional balls of color $i$ in the urn; 
balls are drawn uniformly at
random from the urn, and each drawn ball is replaced together with a new
ball of the same colour. 
Note that then $n_i$ is the number of times colour $i$ has been drawn.
(It does not
matter whether $\ga$ is an integer or not; the extension to non-integer
$\ga$ causes no mathematical problem, 
see \eg{} \cite[Remark 4.2]{SJ154}, \cite{SJ169} or \cite{Jirina}.)
The sequence of vertex labels is thus given by the sequence of colours of
the balls drawn from this urn.
It is well-known, by an explicit calculation, see \eg{} \cite{Markov1917}
(where $N=2$),
that this sequence is exchangeable. By de Finetti's theorem it can thus
can be seen
as an \iid{} sequence of colours with a random distribution $\hP$, which
equals the asymptotic colour distribution. Moreover, it is well-known
\cite{JohnsonKotz}
(see also \cite{Markov1917} and \cite{Polya} for $N=2$)
that this asymptotic distribution is a symmetric Dirichlet distribution 
$\DIR(\ga/N,\dots,\ga/N)$, with the density function $c\prod x_i^{\ga/N-1}$
on the $(N-1)$-dimensional simplex 
$\set{(x_1,\dots,x_N)\in\bbRp^N:\sum_i x_i=1$}.
Consequently, the multigraph process $\GG_N$ can be obtained by the rank 1
model in \refE{Erank1} with the random probability distribution
$\DIR(\ga/N,\dots,\ga/N)$. 

Alternatively, by \refT{Tconfig}, $\GGm$ may be obtained by the
configuration model, with vertex degrees given by the first $2m$ draws in
the \Polya{} urn process described above.

See further \cite{SJ-Warnke}.
\end{example}

\subsection{Rank 1 simple graphs}\label{Srank1simple}
We will in the following sections
study several examples of the simple random graphs $\Gm$
in the rank 1 case.
We note here  a few general formulas.
We ignore the trivial case when the probability distribution \set{q_i} 
is supported on one point. (Then $\Gt$ and $\Gm$ have only a single vertex
and no edges.
In fact, the interesting case is when the support of \set{q_i} is infinite.)
We thus assume $\max q_i<1$. 

Since we ignore loops when constructing the simple graphs $\Gt$ and $\Gm$,
we modify \eqref{muij1} by taking $\mu_{ii}=0$, see \refR{Rloop1}; 
this changes \eqref{mui1} to
$\mu_i=2q_i-2q_i^2$, but we still have $\mu_i\asymp q_i$.
Thus \eqref{vt} and \eqref{et} yield
\begin{align}
  v(t) &\asymp \sum_i \bigpar{1\land(tq_i)},\label{vt1}\\
  e(t) &\asymp \sum_{i\neq j} \bigpar{1\land(tq_iq_j)}
.\label{et1}
\end{align}
Moreover, adding the diagonal terms to the sum in \eqref{et1} does not harm,
since if we assume as we may that $q_1,q_2>0$, then $q_i^2=O(q_1q_i)$ and
$q_1^2=O(q_1q_2)$, and thus 
$\sum_{i} \bigpar{1\land(tq_i^2)} =O\bigpar{\sum_{i>1} \bigpar{1\land(tq_1q_i)}}
=O\bigpar{e(t)}$. Hence also 
\begin{align}
  e(t) &
\asymp \sum_{i, j} \bigpar{1\land(tq_iq_j)}
\asymp \sum_i v(tq_i)
.\label{et11}
\end{align}
Note that although we are interested in large $t$, the argument $tq_i$ in
\eqref{et11} is small for large $i$, so \eqref{et11} requires that we
consider $v(t)$ for both large and small $t$.

Similarly, the expected degree of vertex $i$ in $\Gt$ is
\begin{equation}\label{ED1}
  \begin{split}
  \E D_i &= \sum_{j\neq i} \bigpar{1-e^{-2tq_iq_j}}
\asymp \sum_{j\neq i} \bigpar{1\land \xpar{tq_iq_j}}
\asymp \sum_{j} \bigpar{1\land \xpar{tq_iq_j}}
\asymp v(tq_i).    
  \end{split}
\end{equation}

\section{Dense examples}\label{Sdense}

We may obtain  examples where $\Gm$ and $\Gt$ are dense 
by letting $\mu_{ij}$ decrease very rapidly.

We begin with an extreme case, which gives complete graphs.

\begin{example}[Complete graphs]\label{Ecomplete}
Let $\mu=(\mu\xij)$ be such that
for every \hbox{$k\ge2$}
\begin{equation}
  \label{d1}
0<
\sup_{\ell\ge1}\mu_{k+1,\ell} \le k^{-4}\min_{i<k}\mu_{k,i}.
\end{equation}
(This is not best possible, and may easily be improved somewhat,
but we only want to give a class of examples. To
find necessary and sufficient conditions for $\Gm$ to be (almost) complete 
for large $m$
is an open
problem.) 

For example, we may take $\mu\xij=\xpar{(i\lor j)!}^{-4}$,
or the rank 1 example $\mu\xij=q_iq_j$ with $q_i=\exp\xpar{-3^i}$.

Define $a_i:=\sup_{j}\mu_{ij}$. 
Then \eqref{d1} implies, for every $k\ge2$,
\begin{equation}\label{d2}
a_{k+1}\le k^{-4}\mu_{k1}\le k^{-4} a_k. 
\end{equation}
In particular, for $k\ge2$, $a_{k+1}\le \frac12a_k$.
Moreover, $(k-1)^4a_k \ge  a_{k}\ge k^4a_{k+1}$; hence the sequence
$k^4a_{k+1}$ is decreasing for $k\ge1$.

Define $t_n:=(n^3 a_{n+1})\qw$. 
Let $Y_n$ be the number of edges in $\GG_{t_n}$ with at least one endpoint
outside $[n]$.
Then, since \eqref{d2} implies that $(k+1)a_{k+1}\le \frac12ka_k$ when $k\ge2$,
\begin{equation}
  \begin{split}
	\E Y_n =\sum_{k\ge n+1}\sum_{i<k} t_n\mu_{ki}
<t_n\sum_{k\ge n+1}k a_k
\le 2t_n (n+1) a_{n+1}
=\frac{2(n+1)}{n^3}.
  \end{split}
\end{equation}
Consequently, by Markov's inequality and the Borel--Cantelli lemma, \as{}
$Y_n=0$ for all large $n$.

On the other hand, if $Z_n$ is the number of pairs $(i,j)$ with $i<j\le n$
such that $ij$ is not an edge of $\GG_{t_n}$, \ie, $N_{ij}(t_n)=0$, then
\begin{equation}\label{d4}
  \E Z_n = \sum_{i<j\le n} \P(N_{ij}(t_n)=0)
= \sum_{i<j\le n} e^{-t_n\mu_{ij}}.
\end{equation}
Moreover, if $i<j\le n$, then by \eqref{d1} and \eqref{d2},
$\mu\xij\ge j^4 a_{j+1}\ge n^4a_{n+1}$ and thus
$t_n\mu\xij\ge t_n n^4a_{n+1}=n$.
Hence, \eqref{d4} yields $\E Z_n \le \binom n2 e^{-n}$,
and we see, by the Borel--Cantelli lemma again, 
that \as{} also $Z_n=0$ for all large $n$.

We have shown that \as{} for all large $n$, $\GG_{t_n}$ contains at least
one edge $ij$ whenever $i<j\le n$, but no other edges; in other words, the
simple graph $\G_{t_n}$ is the complete graph $K_n$.
Since $K_n$ has $\binom n2$ edges, this also means that $G_{\binom n2}=K_n$.

We have shown that \as{}, for all large $m$, $G_m$ is the complete graph
$K_n$ if $m=\binom n2$; since $G_n$ is an increasing sequence of graphs, 
it follows that for intermediate values $m=\binom n2+\ell$, $1\le \ell<n$,
$G_m$ consist of $K_n$ plus an additional vertex joined to $\ell$ of the
other vertices. We thus have a complete description of the process $(\Gm)$
for large $m$. (And thus also of the process $\Gt$.)

In particular, for all large $m$, $\Gm$ differs from the complete graph
$K_n$ with $n=v(\Gm)$ by less than $n$ edges, and thus,
see \refS{Slim},
$\dcut(\Gm,K_n)\le\norml{W_{\Gm}-W_{K_n}}\le 2/n=o(1)$.
It follows that in the sense of graph limit theory,
$\Gm\to \gG_1$ \as, where $\gG_1$ is the graph limit defined as the limit of 
the complete graphs,
which is the graph limit defined by the constant
graphon $W_1(x,y)=1$ (on any probability space
$\gO$).
\end{example}

Here is another example, where the limit is less trivial.

\begin{example}\label{Egeo}
Consider a rank 1 example $\mu\xij=q_iq_j$, $i\neq j$,
where $q_i$ has a geometric decay $q_i\asymp  b^{-i}$ for some 
$b>1$.

Let $n\ge 1$ and suppose $b^n\le t\le b^{n+1}$.
Then the 
expected number of edges
$ij$ in $\Gt$ with $i+j>n$ is
at most, 
with $C:=\sup_i b^iq_i<\infty$ and letting $\ell=i+j$, 
\begin{equation}
\sum_{i+j>n}t q_iq_j
\le t\sum_{i+j>n} C^2 b^{-i}b^{-j}
\le
  b^{n+1}\sum_{\ell\ge n+1}\ell C^2 b^{-l}
=O(n).
\end{equation}
Similarly, the
 expected number of edges $ij$ with $i+j\le n$ \emph{not} in $\Gt$
is at most, for $c:=\inf_i b^iq_i>0$,
\begin{equation}
  \begin{split}
\sum_{i+j\le n}\exp\xpar{-2t q_iq_j}
&\le
  \sum_{2\le \ell\le n}\sum_{i=1}^{\ell-1} \exp\bigpar{-tc^2b^{-\ell}}
\le n
  \sum_{2\le \ell\le n} \exp\bigpar{- c^2b^{n-\ell}}
\\&
=O(n).    
  \end{split}
\end{equation}
Moreover, the same argument shows that the expected number of edges $ij$ in
$\Gt$ with $i+j>n+n^{0.1}$ 
and the number of non-edges $ij$ with $i+j<n-n^{0.1}$ both are 
$O(n b^{-n^{0.1}})$; hence the Borel--Cantelli lemma shows that \as{}
for every large $n$ and every $t\in[b^n,b^{n+1}]$, 
$\Gt$ contains every edge with $i+j< n-n^{0.1}$ and 
no edge with $i+j> n+n^{0.1}$; a consequence, we also have
$[n-n^{0.1}-1]\subseteq V(\Gt)\subseteq[n+n^{0.1}]$.
It follows that
if $H_n$ is the graph with vertex set 
\set{1,\dots,n} and edge set \set{ij:i+j\le n},
then \as{} the cutdistance $\dcut(\Gt,H_n)=o(1)$,
when $b^n\le t\le b^{n+1}$.
As \ntoo, 
$H_n\to\gGq$, the graph limit defined by the graphon
$W(x,y)=\indic{x+y\le1}$ on 
$\oi$ (known as the ``half-graphon'').
Consequently, 
$\G_{t}\to\gGq$ \as{} as \ttoo.
By \refP{Pt}, $G_m\to\gGq$ \as{} as \mtoo.
\end{example}

\begin{example}  \label{Egeo'} 
\refE{Egeo} can be generalized without difficulty.
Consider, for example, a rank 1 case $\mu\xij=q_iq_j$ with 
\begin{equation}
  \label{egeo'}
q_i=\exp\bigpar{-ci+O(i^{1-\eps})}
\end{equation}
for some constants $c>0$ and $\eps>0$.
Arguing as in \refL{Egeo} we see that \as, for every large $n$ and all
$t\in[e^{cn},e^{c(n+1)}]$, 
$\Gt$ contains all edges $ij$ with $i+j<n-n^{1-\eps/2}$ and no
edges $ij$ with $i+j>n+n^{1-\eps/2}$. Consequently, \as,
$\dcut(\Gt,H_n)=o(1)$ and thus $\Gt\to\gGq$ as \ttoo{} and $\Gm\to\gGq$ as
\mtoo. 
\end{example}

\begin{example}\label{Ehollywood0}
Consider the simple graphs $\Gt$ and $\Gm$ given by the Hollywood model in
\refE{Ehollywood} in the case $\ga=0$. As shown there, the resulting random
graphs are the same as the ones given by the rank 1 model with a random
probability distribution $(q_i)_1^\infty$ having the distribution
$\GEM(0,\gth)$, where $\gth\in(0,\infty)$ is a parameter.

By a well-known characterization of the $\GEM$ distribution, see
\cite[Theorem 3.2]{Pitman}, this means that
\begin{equation}
  (q_1,q_2,\dots) = \bigpar{(1-X_1),X_1(1-X_2),X_1X_2(1-X_3),\dots},
\end{equation}
where $X_i\sim\Beta(\gth,1)$ are i.i.d. 
In other words, $q_i=(1-X_i)\prod_{j=1}^{i-1}X_j$, and thus
\begin{equation}
  \log( q_i) = \log(1-X_i) + \sum_{j=1}^{i-1}\log(X_j).
\end{equation}
Hence, by the law of iterated logarithm, \as{}
\begin{equation}
  \log(q_i) = -ci + O\bigpar{\sqrt{i\log\log i}}=-ci+O\bigpar{i^{0.6}},
\end{equation}
where $c:=-\E\log(X_1)=1/\gth$. 
Hence, by conditioning on
$(q_i)_1^\infty$, 
\refE{Egeo'} applies.
Consequently,
$\Gm\to\gGq$ \as{} as \mtoo{} for the Hollywood model with $\ga=0$ and any
$\gth>0$. 
\end{example}

\begin{example}\label{Esuper}
For another generalization of \refE{Egeo}, 
consider the rank 1 case with $q_i\asymp \exp\bigpar{-ci^\gam}$ for some
$c>0$ and 
$\gam>0$. It follows by a similar argument that \as{} $\Gm\to W$, where $W$
is the graphon $\indic{x^\gam+y^\gam\le 1}$ on $\oi$.
\end{example}

In Examples \ref{Ecomplete}--\ref{Esuper}, $\Gm$ converges \as{} to some
graph limit. 
There are also many examples, see \eg{} Sections
\ref{Ssparse}--\ref{SXsparse}, for which $\Gm$ are sparse, which is
equivalent to $\Gm\to\gG_0$, the zero 
graph limit defined by the graphon $W(x,y)=0$.
In fact, any graph limit can occur as a limit of $\Gm$, 
at least along a subsequence. Moreover, 
the following result shows that there exists
a ``chameleon'' example where every graph limit occurs as the limit of some
subsequence. 
(Note that this includes that there is a subsequence converging to the zero
graph limit $\gG_0$, 
which means that $e(G_m)=o(v(G_m)^2)$ along this subsequences; hence this
example is neither dense nor sparse.)

\begin{theorem}\label{Tchameleon}
  There exists a matrix $\mu=(\mu\xij)$ such that \as{}
the graphs $\Gm$ are dense in the space of graph limits, in the sense that
for every graph limit $\gG$, there exists a subsequence $G_{m_\ell}$ that
converges to $\gG$.
\end{theorem}

\begin{proof}
  Let $F_k$, $k\ge1$, be an enumeration of all finite (unlabelled) simple
  graphs without isolated vertices, each repeated an infinite number of
  times.
Let $v_k:=v(F_k)$
and let $f_k(i,j)$ be the adjacency matrix of $F_k$.

Let  $\ellN_0:=1$ and, inductively,
$\ellN_k:= k v_k \ellN_{k-1}$ for $k\ge1$.
Let also
\begin{equation}\label{ak}
  a_k:=\prod_{j=1}^k \ellN_j^{-4}.
\end{equation}
Clearly, $\ellN_k\ge k!$, $\ellN_k\ge \ellN_{k-1}$ and $a_k\le a_{k-1}$.
Finally let, for $i\neq j$,
\begin{equation}\label{muk}
  \mu\xij = a_k f_k\Bigpar{\Bigceil{\frac{i}{k\ellN_{k-1}}},
\Bigceil{\frac{j}{k\ellN_{k-1}}}},
\qquad \ellN_{k-1}< i \lor j \le \ellN_k.
\end{equation}
Let $I_k:=[1,\ellN_k]$ and divide $I_k$ into the $v_k$ subintervals 
$I_{k,\ell}:=[(\ell-1)k\ellN_{k-1}+1,\ell k\ellN_{k-1}]$, $\ell=1,\dots v_k$.
Note that \eqref{muk} says that if $i\in I_{k,p}$ and $j\in I_{k,q}$ and not
both $i,j\in I_{k-1}$, then $\mu\xij=a_k f_k(p,q)$.

Let $t_k:=\ellN_k a_k\qw$.
If $n>k$, then the expected number of edges $ij$ in $\GG_{t_k}$ with $i\lor
j\in I_n\setminus I_{n-1}$ is at most, using \eqref{ak},
\begin{equation}
  t_k\sum_{i\lor j\in I_n\setminus I_{n-1}}\mu\xij
\le t_ka_n \ellN_n^2
\le t_k \frac{a_{n-1}}{\ellN_{n}^{2}}
=\frac{a_{n-1}}{a_k}\frac{ \ellN_k}{\ellN_{n}^{2}}
\le \frac{1}{\ellN_{n}}\le\frac{1}{n!}.
\end{equation}
Hence the probability that $\GG_{t_k}$ contains some edge with endpoint not
in $I_k\times I_k$ is at most
\begin{equation}
  \sum_{n>k}\frac{1}{n!}\le \frac{1}{k!},
\end{equation}
and by the Borel--Cantelli lemma, \as{} this happens for only finitely many
$k$.

Similarly, if $(i,j)\in I_k^2\setminus I_{k-1}^2$, then
$\mu\xij\in\set{0,a_k}$, and the probability that there exists some such
pair $(i,j)$ with $\mu\xij=a_k$ but $N\xij(t_k)=0$ is at most
\begin{equation}
 \ellN_k^2 e^{-t_ka_k} =\ellN_k^2 e^{-\ellN_k} =O\bigpar{\ellN_k\qw}=O\bigpar{k!\qw}.
\end{equation}
Consequently,
again by the Borel--Cantelli lemma, \as{} for every large $k$, there exists
no such pair $(i,j)$.

We have shown that \as{} for every large $k$, the simple graph $\G_{t_k}$
contains no edge with an endpoint outside $I_k$, and for 
$(i,j)\in I_k^2\setminus I_{k-1}^2$,
recalling \eqref{muk},
if $i\in I_{k,p}$ and $j\in I_{k,q}$, then there is an edge $ij$ if and only
if $f_k(p,q)=1$. 
In particular, since $F_k$ has no isolated vertices, every $i\in I_k$ is the
endpoint of some edge in $\G_{t_k}$ and thus a vertex, but no $i\notin I_k$
is;
in other words, \as{} for every large $k$,
$V(\G_{t_k})=I_k$.
It follows that if $F_k^*$ is the blow-up of $F_k$ with
every vertex repeated $kN_{k-1}$ times, then \as{} for every large $k$,
the graphs $\G_{t_k}$ and $H^*_k$ have the same vertex set $I_k$ and their 
adjacency matrices can differ only for $(i,j)\in I_{k-1}^2$.
Consequently, using \refR{Rblowup},
\begin{equation}\label{qb}
  \begin{split}
  \dcut(\G_{t_k},F_k)
&=
  \dcut(\G_{t_k},F_k^*)
\le   \cn{W_{\G_{t_k}}-W_{F_k^*}}
\le   \norml{W_{\G_{t_k}}-W_{F_k^*}}
\\&
\le \frac{N_{k-1}^2}{N_k^2}
=\frac1{(kv_k)^2}
\le k^{-2},
  \end{split}
\end{equation}
\as{} for all large $k$.

Now, let $\gG$ be a graph limit, and let $\ell\ge1$.
By graph limit theory (or definition), there exists a sequence of graphs
$\xxH_j$ with $v(\xxH_j)\to\infty$ and $\dcut(\xxH_j,\gG)\to0$ as $j\to\infty$; 
hence we may
take $j$ so large that $\xxH:=\xxH_j$ satisfies $v(\xxH)>\ell$ and
$\dcut(\xxH,\gG)<1/\ell$. $\xxH$ may have isolated vertices, so we define
$\xxH'$ by  
choosing a vertex $v\in \xxH$ and adding an edge from $v$ to any other vertex
in $\xxH$. Then at most $v(\xxH)-1$ edges are added, and thus, similarly to
\eqref{qb}, 
  \begin{equation}\label{qc}
  \begin{split}
  \dcut(\xxH',\xxH)
\le   \cn{W_{\xxH'}-W_{\xxH}}
\le   \norml{W_{\xxH'}-W_{\xxH}}
\le \frac{2 v(\xxH)}{v(\xxH)^2}
<\frac{2}\ell.
  \end{split}
\end{equation}
Moreover, $\xxH'$ has no isolated vertices, and thus $\xxH'$ occurs infinitely
often in the sequence $(F_k)$ above.
Consequently, a.s., there exists $k>\ell$ such that \eqref{qb} holds and
$F_k=\xxH'$. Then, by \eqref{qb} and \eqref{qc},
\begin{equation}
  \dcut(\G_{t_k},\gG)
\le\dcut(\G_{t_k},F_k)+\dcut(\xxH',\xxH)+\dcut(\xxH,\gG)
<\frac{1}{k^{2}}+\frac{2}\ell+\frac{1}\ell<\frac{4}\ell.
\end{equation}
By \refP{Pt}, this means that if $m_\ell:=N(t_k)$, then
$\dcut(G_{m_\ell},\gG)<4/\ell$. 
Since $\ell$ is arbitrary, this completes the proof.
(We may choose $m_\ell$ inductively, and choose $k$ above so large that
$m_\ell>m_{\ell-1}$.) 
\end{proof}

The chameleon example in \refT{Tchameleon} is theoretically very
interesting, but it is hardly useful as a model in applications; since the
behaviour of $\Gm$ 
changes so completely with $m$, it is a model of nothing rather than a model 
of everything.

If we want convergence of the full sequence $\Gm$ and not just 
subsequence convergence as in \refT{Tchameleon}, 
we do not know whether any graph limit can occur as a limit.

\begin{problem}
  For which graph limits $\gG$ does there exist a matrix $(\mu\xij)$ such
  that for the corresponding simple random graphs, $\Gm\to\gG$?
\end{problem}

\section{Sparse examples}\label{Ssparse}

We gave in the preceding section some dense examples. It seems to be more
typical, however, that the graph $\Gm$ contains many vertices of small
degree (maybe even degree 1), and that the graph  is sparse.
We give here a few, related, rank 1 examples; see also the following section.

\begin{example}\label{Epower}
Consider a rank 1 example with $q_i\asymp i^{-\gam}$  for some $\gam>1$.
Then \eqref{vt1} yields
\begin{equation}\label{Epv}
  v(t)\asymp\sum_{i\ge1}\bigpar{1\land(ti^{-\gam})}
=\sum_{i\le t^{1/\gam}} 1 + t\sum_{i> t^{1/\gam}} i^{-\gam}
\asymp 
\begin{cases}
t^{1/\gam},&t\ge1,  \\
t,&t<1.
\end{cases}
\end{equation}
This yields by \eqref{et11}, for $t\ge2$,
\begin{equation}
  \begin{split}
  e(t)\asymp \sum_{i\ge1}v(tq_i)
\asymp \sum_{i\ge1}v(ti^{-\gam})
\asymp \sum_{i\le t^{1/\gam}} t^{1/\gam}i^{-1} + \sum_{i> t^{1/\gam}}
ti^{-\gam}
\asymp t^{1/\gam}\log t.    
  \end{split}
\end{equation}
Hence, using \refT{Too}, \as{}
$v(\Gt)\asymp t^{1/\gam}$ and $e(\Gt)\asymp t^{1/\gam}\log t$ as \ttoo,
and
$v(\Gm)\asymp m^{1/\gam}$ and $e(\Gm)\asymp m^{1/\gam}\log m$ as \mtoo.
It follows that the average degree in $\Gm$ is $\asymp \log m$.
\end{example}

In this example we may also show that the degree distribution has a
power-law; we state this as a theorem. 
There is no standard precise definition of what is meant by a power-law
degree distribution; we may say that a
random variable $X$ has a power-law distribution with exponent $\tau$ if
$\P(X>x)\asymp x^{-(\tau-1)}$ as \xtoo, but this does not make sense for the
degree distribution of a finite graph,
so we must either consider the
asymptotic degree distribution, provided one exists, or give uniform
estimates for a suitable range of $x$. 
(See \eg{} \cite[Sections 1.4.1 and 1.7]{RemcoI} for a discussion of
power-laws for degree distributions.)
We follow here the second possibility.

For a (finite) graph $G$, let $v\gek(G)$ be the number of vertices of
degree at least $k$, and let $\pi\gek(G):=v\gek(G)/v(G)$, the probability that
a random vertex has degree $\ge k$.

\begin{theorem}\label{Tpower}
In \refE{Epower}, the random graphs $\Gm$ have a power-law distribution
with exponent 2 in the following sense.
There exist positive constants $c$ and $C$ such that \as{} for every large $m$,
\begin{align}
 && \pi\gek(\Gm) &\le C/k, && 1\le k<\infty, &&\label{pow<}
  \\
&&  \pi\gek(\Gm) &\ge c/k, && 1\le k\le c v(G_m).&&\label{pow>}
\end{align}
\end{theorem}
As usual, the same result holds for $\Gt$.
Note that the restriction $k\le cv(\Gm)$ in \eqref{pow>} is necessary, and
best possible (up to the value of the constants); we necessarily have
$\pi\gek(G)=0$ when $k\ge v(G)$.
Note also that we have the same exponent $\tau=2$ for every $\gam>1$.

\begin{proof}\CCreset\ccreset
As usual, we prove the results for $\Gt$; the results for $\Gm$  follow
by \refP{Pt}. We then can write \eqref{pow<}--\eqref{pow>} as
$v\gek(\Gt)\le C v(\Gt)/k$, $k\ge1$, and
$v\gek(\Gt)\ge c v(\Gt)/k$, $1\le k\le cv(\Gt)$, and
by \refT{Too} and \eqref{Epv}, it suffices (and is equivalent) to prove
that \as{}
\begin{align}
 && v\gek(\Gt) &\le \CC t\qgam/k, && 1\le k<\infty, &&\label{pow<t}
  \\
&&  v\gek(\Gt) &\ge \cc t\qgam/k, && 1\le k\le \cc t\qgam,&&\label{pow>t}
\end{align}
for every large $t$.

Let $I\xij$ be the indicator of an edge  $ij$ in $\Gt$; thus
$I\xij\sim\Be\bigpar{1-e^{-2tq_iq_j}}$. 
Let $D_i:=\sum_{j\neq i}I\xij$ be the degree of $i$ in the simple graph $\Gt$.
(The degree is defined as 0 if $i$ is not a vertex.)

\step{(i)}{The upper bound \eqref{pow<t}} 
We fix $t\ge1$ and an integer $k\ge1$; for convenience we often omit them
from the notation, but note that 
many variables below depend on them, while
all explicit and implicit constant are
independent of $t$ and $k$.

Let $J_i:=\indic{D_i\ge k}$ and $\NN:=\sum_i J_i=v\gek(\Gt)$.

Let $A$ be a large constant, chosen later, 
and assume that $k\ge A$,
let $i_0:=At^{1/\gam}/k$
and let
$\NNx:=\sum_{i>i_0}J_i$.
Thus
$N\le \NNx+i_0$.

If $i\ge i_0$, then
using \eqref{ED1}, \eqref{vC} and \eqref{Epv},
\begin{equation}\label{edd}
\E D_i\asymp v(tq_i)\asymp v(ti^{-\gam})
\le v\bigpar{A^{-\gam}k^\gam}\asymp k/A. 
\end{equation}
Thus $\E D_i \le \CC k/A$ for some $\CCx\ge0$, and choosing 
$A=\max(14\CCx,4)$, we find
that 
$\E D_i \le k/14\le (k-1)/7$. Since $D_i$ is a sum $\sum_j I\xij$
of independent Bernoulli variables, a Chernoff bound (see \eg{}
\cite[(2.11) and Theorem 2.8]{JLR}) yields 
\begin{equation}
  \label{xc}
\E J_i=\P(D_i\ge k)\le e^{-k},
\qquad i\ge i_0,
\end{equation}
and also, for later use,
\begin{equation}
  \label{xb}
\P(D_i\ge k-1)\le e^{1-k},
\qquad i\ge i_0.
\end{equation}
For $i\ge t^{1/\gam}$ we also have, by \eqref{edd} and \eqref{Epv},
\begin{equation}\label{edda}
  \E D_i \asymp v(ti^{-\gam}) \asymp ti^{-\gam}.
\end{equation}
Let $(x)_r:=x(x-1)\dotsm(x-r+1)$, the falling factorial.
Since $D_i$ is a sum of independent indicators, 
it is easily seen that for any positive integer $r$, 
the factorial moment can be bounded by
$\E (D_i)_r\le(\E D_i)^r$. Hence, by \eqref{edda} and 
Markov's inequality, since we assume
$k\ge A\ge 4$,
\begin{equation}\label{eddb}
  \E J_i = \P(D_i\ge k) \le \frac{\E (D_i)_4}{(k)_4}
\le \frac{(\E D_i)^4}{(k)_4}
\le \CC \frac{(ti^{-\gam})^4}{k^4}
\le \CCx \frac{ti^{-\gam}}{k^4},
\qquad i\ge t^{1/\gam}.
\CCdef\CCeddb
\end{equation}
(This also follows from \cite[(2.10) and Theorem 2.8]{JLR}.)
Summing \eqref{xc} and \eqref{eddb}, we obtain
\begin{equation}\label{en'}
  \E \NNx 
= \sum_{i> i_0} \E J_i
\le \sum_{i_0< i\le t^{1/\gam}} e^{-k} + \sum_{i> t^{1/\gam}} \CCeddb ti^{-\gam}/k^4
\le \CCname{\CCen} t^{1/\gam}/k^4.
\end{equation}

For the variance of $\NNx$, we note that
the indicators $J_i$ are not quite independent, since an edge $ij$
influences both $J_i$ and $J_j$, but conditioned on $I\xij$, $J_i$ and $J_j$
are independent. Hence, for any distinct $i$ and $j$,
\begin{equation*}
  \begin{split}
\E(J_iJ_j)
&=\P(I\xij=1)\E\bigpar{J_iJ_j\mid I\xij=1} 
+ \P(I\xij=0)\E\bigpar{J_iJ_j\mid I\xij=0} 
\\
&=\P(I\xij=1)\E \bigpar{J_i\mid I\xij=1}\E\bigpar{J_j\mid I\xij=1} 
\\&\hskip2em{}
+ \P(I\xij=0)\E\bigpar{J_i\mid I\xij=0} \E\bigpar{J_j\mid I\xij=0} 
\\
&\le \P(I\xij=1)\P(D_i\ge k-1)\P(D_j\ge k-1)
+ \P(I\xij=0)\E{J_i} \E {J_j}
  \end{split}
\end{equation*}
and thus
\begin{equation}\label{xa}
  \Cov(J_i,J_j)\le \P(I\xij=1)\P(D_i\ge k-1)\P(D_j\ge k-1).
\end{equation}
By \eqref{xa} and \eqref{xb}, for $i,j\ge i_0$ with $i\neq j$,
\begin{equation}
  \Cov(J_i,J_j) \le 2tq_iq_j e^{2(1-k)}
\le \CC t i^{-\gam}j^{-\gam}e^{-2k}.
\end{equation}
Consequently, using also \eqref{en'},
\begin{equation}
  \begin{split}
  \Var \NNx 
&=\sum_{i,j> i_0}\Cov(J_i,J_j) 
\le \E \NNx+  \CCx te^{-2k} \sum_{i,j> i_0}i^{-\gam}j^{-\gam}
\\&
\le \CCen t^{1/\gam}k^{-4} + \CC te^{-2k} i_0^{2(1-\gam)} 
\le \CCname\CCche t^{1/\gam}k^{-4}.
  \end{split}
\end{equation}
Hence, by Chebyshev's inequality,
\begin{equation}\label{che}
  \P\bigpar{\NNx-\E \NNx>t^{1/\gam}/k}
\le \frac{\Var \NNx}{(t^{1/\gam}/k)^2} \le \CCche t^{-1/\gam}k^{-2}.
\end{equation}

We have so far kept $t$ and $k$ fixed. We now sum \eqref{che} over all $k\ge
A$ and $t=2^\ell$ for $\ell\in \bbN$, and find by the Borel--Cantelli lemma that
\as{} for every large $t$ of this form and every $k\ge A$, 
$\NNx-\E \NNx\le t^{1/\gam}/k$, and
consequently, using also \eqref{en'},
\begin{equation}
  \NN \le \NNx+i_0 \le \E \NNx+ t^{1/\gam}/k+i_0
\le \CC t^{1/\gam}/k.
\end{equation}
This is \eqref{pow<t} for $k\ge A$ and
$t\in\set{2^\ell}$;
since $\NN$ increases with $t$, 
\eqref{pow<t} follows in general (with a different constant), 
\as{} for large $t$ and all $k\ge A$.

For $k<A$, \eqref{pow<} and \eqref{pow<t} follow trivially from
$v\gek(\Gt)\le v(\Gt)$. 

\step{(ii)}{The lower bound \eqref{pow>t}} 
Fix again $t\ge1$ and $k\ge1$, 
let $B$ be a large constant chosen later,
and assume that $k\le t^{1/\gam}/B$.

Let $L$ be the set of odd integers $i$ with $1\le i \le i_1:=B\qw t^{1/\gam}/k$,
and let $R$ be the set of even integers $j$ with $1\le j\le 6k$.
By our assumption on $k$, $i_1\ge1$, and thus $|L|=\floor{(i_1+1)/2}\ge i_1/3$.
Note that the indicators $\set{I\xij}_{i\in L,\,j\in R}$ are independent.
For $i\in L$, let $D_i':=\sum_{j\in R} I\xij$ and $J'_i=\indic{D_i'\ge k}$.
Thus the indicators $\set{J_i'}_{i\in L}$ are independent.
Also, let $N':=\sum_{i\in L}J_i'$. Since $J_i' \le J_i$, we have
$N'\le \sum_{i\in L} J_i \le \sum_{i\ge1}J_i=N=v\gek(\Gt)$.

If $i\in L$ and $j\in R$, then
$ij\le 6ki_1=6B\qw t\qgam$, and thus
\begin{equation}\label{krk}
  tq_iq_j \ge \cc t i^{-\gam}j^{-\gam}
\ge \cc B^\gam.
\end{equation}
Choose $B:=\ccx^{1/\gam}$; then by \eqref{krk}, when $i\in L$ and $j\in R$,
$tq_iq_j\ge1$ and thus
\begin{equation}\label{krx}
  \P(I\xij=0) = e^{-2tq_iq_j} \le e^{-2}.
\end{equation}
Since $|R|=3k$, it follows that if $i\in L$, then
$\E D_i' \ge 3(1-e^{-2})k>2.5k$, and moreover, by a Chernoff bound
(\eg{} \cite[(2.12)]{JLR}),
\begin{equation}\label{kry}
\P(J_i'=0)=  \P(D_i'<k) \le e^{- k} \le e^{-1}.
\end{equation}

Since the indicators $J_i'$ are independent for $i\in L$, 
another  Chernoff bound shows that
\begin{equation}\label{krz}
\P(N'<|L|/2) \le e^{-\cc |L|} \le e^{-\cc i_1}.
\end{equation}
Alternatively, \eqref{kry} and a union bound yield
\begin{equation}\label{krzz}
\P(N'<|L|/2)  \le
\P(N'<|L|)\le \sum_{i\in L} \P(J_i'=0) \le i_1 e^{-k}.
\end{equation}
If $1\le k \le t^{1/2\gam}$, then $i_1\ge B\qw t^{1/2\gam}$, and thus
\eqref{krz} yields $\P(N'<|L|/2) \le e^{-\cc t^{1/2\gam}}$.
If $t^{1/2\gam}< k\le t^{1/\gam}/B$, then \eqref{krzz} yields 
$\P(N'<|L|/2) \le i_1 e^{-\cc t^{1/2\gam}}\le \CC e^{-\cc t^{1/2\gam}}$.
Consequently, for every $k\le t^{1/\gam}/B$,
\begin{equation}\label{krw}
  \P(N< |L|/2)\le 
\P(N'<|L|/2) \le  \CC e^{-\cc t^{1/2\gam}}.
\end{equation}

We have kept $k$ and $t$ fixed, but we now sum \eqref{krw} over all $k\le
t^{1/\gam}/B$ and $t=2^\ell$ for some $\ell\in\bbNo$. It follows by the
Borel--Cantelli lemma that \as{} for every large $t$ of this form 
and every $k\le t^{1/\gam}/B$,
$N\ge |L|/2\ge i_1/6 \ge \cc t\qgam/k$.
This proves \eqref{pow>t} for $t$ of the form $2^\ell$, and again the
general case follows since $N$ is monotone in $t$.
\end{proof}

Furthermore, assuming $q_i\sim ci^{-\gam}$ 
in \refE{Epower}, we can show that
the empirical stretched graphon, with a
suitable stretch, converges to a (non-integrable) graphon on $\bbRp$, in the
sense discussed in \refS{Sgraphex}.

\begin{theorem}\label{Tgraphex}
  In \refE{Epower}, assume that $q_i\sim c i^{-\gam}$ as $i\to\infty$, with
  $c>0$. 
Then the stretched empirical graphon $W_{\Gt,t^{1/2\gam}}\togp W$ \as{} as
\ttoo, where $W$ is the graphon $W(x,y)=1-\exp\bigpar{-2c^2 x^{-\gam}y^{-\gam}}$ on
$\bbRp^2$.

As a consequence, $\Gt\togs W$ \as{} as \ttoo.
\end{theorem}

Note that $W(x,y)\ge 1-\exp(-2c^2)>0$ when $xy\le 1$, and thus $\int W=\infty$.

We prove first two lemmas.

\begin{lemma}\label{LS}
Let $\ZZ$ be an array of \iid{} random variables.
Furthermore,
let $x_1,\dots,x_n>0$ be distinct and let $X$ be a  random variable,
independent of the array $\ZZ$, with $X\sim U(a,b)$ where $0<a<b<\infty$.
Then, 
\begin{equation}\label{ls1}
  \cL\bigpar{(Z_{\ceil{tx_i},\ceil{tX}})_{i=1}^n\mid \ZZ}
\asto \cL\bigpar{(Z_{i,n+1})_{i=1}^n}
\end{equation}
 as \ttoo.
\end{lemma}

In other words, conditionally on $\ZZ$
and for \aex{} every realization of $\ZZ$,
the random vector
$(Z_{\ceil{tx_i},\ceil{tX}})_{i=1}^n$ converges in distribution to 
$(Z'_{i,n+1})_{i=1}^n$, 
where $\ZZZ$ is an independent copy of $\ZZ$.

\begin{proof}
It suffices to prove that for every fixed rational
$z_1,\dots,z_n$,
\begin{equation}\label{pik}
  \P\bigpar{Z_{\ceil{tx_i},\ceil{tX}}\le z_i, 1\le i\le n\mid \ZZ}
\asto \pi:=\P\bigpar{Z_{i,n+1}\le z_i, 1\le i\le n},
\end{equation}
where
\begin{equation}\label{pi}
  \pi=\prod_{i=1}^n\P\bigpar{Z_{i,n+1}\le z_i}=\prod_{i=1}^n\P(Z_{11}\le z_i).
\end{equation}
Let further $I_{k,l,i}:=\indic{Z_{kl}\le z_i}$
and $J_l:=\prodin I_{\ceil{tx_i},l,i}$.
Then, with the error term coming from edge effects,
\begin{equation}\label{ls3}
  \begin{split}
& P_t:= \P\Bigpar{Z_{\ceil{tx_i},\ceil{tX}}\le z_i, 1\le i\le n\mid \ZZ}
=\E\Bigpar{\prodin I_{\ceil{tx_i},\ceil{tX},i}\mid\ZZ}
\\&\qquad
=\E\bigpar{J_{\ceil{tX}}\mid\ZZ}
=\frac{1}{(b-a)t} \sum_{ta<l\le tb} J_l+o(1).
  \end{split}
\raisetag\baselineskip
\end{equation}
If $t$ is sufficiently large, then $\ceil{tx_1},\dots,\ceil{tx_n}$ are
distinct, and then, see \eqref{pi}, 
\begin{equation}
\E J_l=\prodin \E I_{\ceil{tx_i},l,i}
=\prodin  \P\bigpar{Z_{\ceil{tx_i},l}\le z_i}=\pi.
\end{equation}
Furthermore, then the variables $J_l\sim\Be(\pi)$ are \iid, so their sum in
\eqref{ls3} has a binomial distribution, and a Chernoff bound shows that for
every $\eps>0$, there is a $c=c(\eps)>0$ such that for large $t$,
\begin{equation}\label{pit}
  \P\bigpar{|P_t-\pi|>\eps}\le e^{-c t}.
\end{equation}
This shows that $P_t$ converges to $\pi$ in probability as \ttoo. 
In order to show convergence \as, we note that if $0<t<u$, and $t(b-a)>1$,
then (for fixed $a$ and $b$) $\P(\ceil{tX}\neq\ceil{uX})=O(u-t)$, and
consequently, for some $C>0$,
\begin{equation}\label{put}
|  P_t-P_u|\le \P(\ceil{tX}\neq\ceil{uX})\le C(u-t).
\end{equation}
Let $\eps>0$, let $N:=\ceil{C/\eps}$ and let $t_n:=n/N$. 
By \eqref{pit} and the Borel--Cantelli lemma, \as{} $|P_{t_n}-\pi|\le\eps$
for all large $n$. Furthermore, if $n$ is large and $t_n\le t\le t_{n+1}$, then
\eqref{put} implies $|P_t-P_{t_n}|\le\eps$, and thus $|P_t-\pi|\le2\eps$.
Consequently, \as, $|P_t-\pi|\le 2\eps$ for every large $t$.
Since $\eps$ is arbitrary, this proves \eqref{pik} and thus the lemma.
\end{proof}

\begin{lemma}\label{LQ}
Let $\ZZ$ be an array of \iid{} random variables, and
let  $(X_1,\dots,X_n)$ be a random vector in $\bbRp^n$ with an absolutely
continuous distribution,
independent of the array $\ZZ$.
Then, 
\begin{equation}\label{lq}
  \cL\bigpar{(Z_{\ceil{tX_i},\ceil{tX_j}})_{1\le i<j\le n}\mid \ZZ}
\asto   \cL\bigpar{(Z_{i,j})_{1\le i<j\le n}}
\end{equation}
 as \ttoo.
\end{lemma}

\begin{proof}
\stepx\label{stepa}
Assume first that $X_1,\dots,X_n$ are independent with $X_i\sim U(I_i)$ for
some  intervals $I_1,\dots,I_n$.  
In this case we prove \eqref{lq} by induction on $n$, so we may assume that
\begin{equation}\label{lq-}
  \cL\bigpar{(Z_{\ceil{tX_i},\ceil{tX_j}})_{1\le i<j\le n-1}\mid \ZZ}
\asto   \cL\bigpar{(Z_{i,j})_{1\le i<j\le n-1}}.
\end{equation}
Furthermore, by \refL{LS} and conditioning on $X_1,\dots,X_{n-1}$,
\begin{equation}\label{lq+}
  \cL\bigpar{(Z_{\ceil{tX_i},\ceil{tX_n}})_{1\le i\le n-1}\mid \ZZ, X_1,\dots,X_{n-1}}
\asto   \cL\bigpar{(Z_{i,n})_{1\le i\le n-1}}.
\end{equation}
The result \eqref{lq} follows by \eqref{lq-} and \eqref{lq+}, which shows
the induction step and completes the proof of this step.

\stepx\label{stepb}
Suppose that there exists a finite family of disjoint intervals $I_k$ such
that
the density function 
$f(x_1,\dots,x_n)$ of $(X_1,\dots, X_n)$ is 
supported on $\bigpar{\bigcup_k I_k}^n$ and
constant on each $\prodin I_{k_i}$.
Then \refStep{stepa} shows that for each sequence $k_1,\dots,k_n$ of
 indices, \eqref{lq} holds conditioned on $(X_1,\dots,X_n)\in
\prodin I_{k_i}$. Hence \eqref{lq} holds unconditioned too.

\stepx\label{stepc}
The general case. Let $f(x_1,\dots,x_n)$ be the density function of
$(X_1,\dots, X_n)$, and let
$\eps>0$. 
Then there exists a density function $f_0(x_1,\dots,x_n)$ of the type in
\refStep{stepb} such that $\int |f-f_0|\dd x_1\dots\dd x_n<\eps$.
We can interpret $f_0$ as the density function of a random vector
$\XX^0=(X^0_1,\dots,X^0_n)$, and we can couple this vector with
$\XX=(X_1,\dots,X_n)$ such that 
$\P\bigpar{\XX\neq\XX^0}<\eps$.

Since \refStep{stepb} applies to $\XX^0$, it follows that
\begin{equation}\label{ls1x}
\P\bigpar{\text{the convergence in \eqref{lq} holds}}
\ge \P\bigpar{(X,Y)=(X_0,Y_0)}>1-\eps.
\end{equation}
Since $\eps>0$ is arbitrary, \eqref{lq} follows.
\end{proof}

\begin{proof}[Proof of \refT{Tgraphex}] \CCreset
Let $w_x:=c\qw q_{\ceil x} x^{\gam}=1+o(1)$, as \xtoo.

We can construct $\Gt$ for all $t>0$ by taking \iid{} random variables
$Z_{kl}\sim\Exp(1)$ and letting there be an edge $kl$ in $\Gt$ if
$2tq_kq_l \ge Z_{kl}$, for every pair $(k,l)$ with $k<l$.

Let $\hwt:=W_{\Gt,t^{1/2\gam}}$ be the stretched empirical graphon in
the statement.
Fix $r>0$, and consider the random graph $G_r(\hwt)$; this is by
\refS{Sgraphex} obtained by taking a Poisson process $\set{\eta_i}_i$ on
$\bbRp$ with intensity $r$ (where we assume $\eta_1<\eta_2<\dots$), 
and then taking an edge $ij$ if and only if
$\hwt(\eta_i,\eta_j)=1$. By the definition of $\hwt$, this is equivalent to
$\Gt$ having an edge between $\ceil{\tgg \eta_i}$ and $\ceil{\tgg \eta_j}$,
and thus by the construction of $\Gt$ to
(assuming that $t$ is large so that $\ceil{\tgg \eta_i}\neq\ceil{\tgg \eta_j}$)
\begin{equation}\label{cf}
 2 t q_{\ceil{\tgg \eta_i}} q_{\ceil{\tgg \eta_j}} \ge Z_{\ceil{\tgg \eta_i},\ceil{\tgg \eta_j}}
\end{equation}
or, equivalently, 
\begin{equation}\label{nt}
  2c^2 \eta_i^{-\gam}\eta_j^{-\gam}
\ge w_{{\tgg \eta_i}}\qw w_{{\tgg \eta_j}}\qw  Z_{\ceil{\tgg \eta_i},\ceil{\tgg \eta_j}}.
\end{equation}

Fix $n<\infty$ and consider the edge indicators $I_{i,j,t}$ in $G_r(\hwt)$ for
$1\le i<j\le n$. Furthermore, fix a large integer $N$ and condition
$(\eta_1,\dots,\eta_n)$ on $\ceil{N\eta_1},\dots,\ceil{N\eta_n}$.
By \refL{LQ}, and recalling $w_x=1+o(1)$,
the distribution of the \rhs{} of \eqref{nt} converges \as{} to
independent $\Exp(1)$ variables, jointly for $1\le i<j\le n$.
Since $I_{i,j,t}$ equals the indicator of \eqref{nt}, it follows by first
replacing the \lhs{} of \eqref{nt} by upper and lower bounds obtained by
rounding each $\eta_i$ down or up to the nearest multiple of $1/N$, applying
\refL{LQ} and then letting \Ntoo, that
\begin{equation}\label{harda}
  \begin{split}
\cL\bigpar{(I_{i,j,t})_{1\le i<j\le n}\mid \Gt}
\to 
\cL\bigpar{\bigpar{\etta
  \bigset{2c^2\eta_i^{-\gam}\eta_j^{-\gam}\ge Z_{ij}}}_{1\le i<j\le n}}.
  \end{split}
\end{equation}
Here, conditioned on $\eta_1,\dots,\eta_n$, the indicators in the \rhs{} are
independent, and have (conditional) expectations
\begin{equation}
\P\bigpar{ 2c^2\eta_i^{-\gam}\eta_j^{-\gam}\ge Z_{ij}\mid \eta_i,\eta_j}
=1-\exp\bigpar{-2c^2\eta_i^{-\gam}\eta_j^{-\gam}}
=W(\eta_i,\eta_j).
\end{equation}
This equals the (conditional) probability of an edge $ij$ in $G_r(W)$.
Consequently, if $I_{i,j}$ is the indicator of an edge $ij$ in $G_r(W)$,
\eqref{harda} shows that \as, as \ttoo,
\begin{equation}\label{hardb}
\bigpar{(I_{i,j,t})_{1\le i<j\le n}\mid \Gt}
\dto
{(I_{i,j})_{1\le i<j\le n}}
\end{equation}
This shows the desired convergence $G_r(\hwt)\dto G_r(W)$, 
provided we restrict the graphs to a fixed finite set of vertices.

To extend this to the infinite number of potential vertices, we need a
tightness argument. (Unfortunately, we did not find a really simple argument.)
Let $a,b>0$, and let $V_{a,b,t}$ denote the number of edges in $G_r(\hwt)$
with endpoints labelled $\eta_i,\eta_j$ with $\eta_i\in(a,2a]$ and
$\eta_j\in(b,2b]$. Then, \cf{} \eqref{cf},
\begin{equation}\label{kabo}
  \E \bigpar{V_{a,b,t}\mid\ZZ}
\le \sum_{
  k=\ceil{a\tgg}}^{\ceil{2a\tgg}}\sum_{l=\ceil{b\tgg}}^{\ceil{2b\tgg}}r^2t^{-1/\gam}
\ett{2tq_kq_l\ge Z_{k,l}}.
\end{equation}
For $k,l$ in the ranges in \eqref{kabo}, $q_k \le \CC a^{-\gam}t^{-1/2}$ and
$q_l \le \CCx b^{-\gam}t^{-1/2}$.
Define $J_{kl}:=\ett{2\CCx^2a^{-\gam}b^{-\gam}\ge Z_{k,l}}$,
$S_{m,n}:=\sum_{k\le m, l\le n} J_{kl}$, $\bS_{m,n}:=S_{m,n}/mn$ and
$\Sx:=\sup_{m,n\ge1} \bS_{m,n}$.
Then \eqref{kabo} implies, assuming $t\ge t_{a,b}:=\max\set{a^{-2\gam},b^{-2\gam}}$,
\begin{equation}\label{kaboh}
  \begin{split}
  \E \bigpar{V_{a,b,t}\mid\ZZ}
&
\le \sum_{ k=\ceil{a\tgg}}^{\ceil{2a\tgg}}\sum_{l=\ceil{b\tgg}}^{\ceil{2b\tgg}}
r^2t^{-1/\gam} J_{kl}
\le r^2 t^{-1/\gam} S_{\ceil{2a\tgg},\ceil{2b\tgg}} 
\\&
\le r^2 t^{-1/\gam} \ceil{2a\tgg}\ceil{2b\tgg} \Sx 
\le
9r^2 ab \Sx.    
  \end{split}
\raisetag\baselineskip
\end{equation}
Fix $p>1$ with $p<\gam$. Then by the multi-dimensional version of Doob's
$L^p$ inequality, see
\cite[Lemma 3]{Smythe},
\eqref{kaboh} implies, for fixed $r$,
\begin{equation}\label{kabohh}
  \begin{split}
\E\sup_{t\ge t_{a,b}}  \E \bigpar{V_{a,b,t}\mid\ZZ}
&\le
\CC ab \E \Sx
\le \CCx ab (\E (\Sx)^p)^{1/p}
\le \CC ab (\E I_{11})^{1/p}
\\&
\le \CC a^{1-\gam/p}b^{1-\gam/p}.
  \end{split}
\raisetag\baselineskip
\end{equation}

Let $\eps>0$, and use \eqref{kabohh} with $a=2^m\eps$ and $b=2^n\eps$. Then
summing over all $(m,n)\in\bbZ_+$ with $m\lor n\ge N$ implies, using
Markov's inequality,
\begin{equation}\label{kabohhh}
  \begin{split}
&\E\sup_{t\ge t_{\eps,\eps}}  
\P\bigpar{\hwt\text{ has an edge labelled 
 $(x,y)\in[\eps,\infty)^2\setminus[\eps,2^N\eps]^2$}\mid\ZZ}
\\&\qquad
\le \CC 2^{-(\gam/p-1)N}\eps^{2(1-\gam/p)}.
  \end{split}
\raisetag\baselineskip
\end{equation}
Choosing $N$ large enough, this is less than $\eps$. Furthermore,
the probability that $\hwt$ has a vertex labelled $<\eps$ is at most
$\P(\eta_1<\eps)<\eps$, and we can choose $n$ such that
$\P(\eta_n\le 2^N\eps)<\eps$.

It now follows from \eqref{hardb} that for any finite graph $H$,
\begin{equation}
\bigabs{\P\bigpar{G_r(\hwt)=H\mid \Gt} -\P(G_r(W)=H)}
\le 3\eps + o(1)  
\end{equation}
\as{} as \ttoo. 
Since $\eps>0$ is arbitrary, this shows $\bigpar{G_r(\hwt)\mid\Gt}\dto
G_r(W)$ \as{} as \ttoo, for every fixed $r<\infty$,
which is the same as $\hwt\togp W$.

Finally, we note that $\hwt\togp W$ implies $\hwt\togs W$, 
see \cite{VR2,SJ317},
and that $\togs$
is not affected by stretchings of the graphons; hence \as{}
also $W_{\Gt}\togs W$,
\ie, $\Gt\togs W$. 
\end{proof}

\begin{example}
  \label{Ehollywood+}
  Consider the simple graphs $\Gt$ and $\Gm$ given by the Hollywood model in
  \refE{Ehollywood} in the case $0<\ga<1$. As shown there, the resulting
  random graphs are the same as the ones given by the rank 1 model with a
  random probability distribution $(q_i)_1^\infty$ having the distribution
  $\PD(\ga,\gth)$, where $\gth>-\ga$ is the second parameter.  This implies
  that \as{} $q_i\sim Zi^{-1/\ga}$ for some (random) $Z>0$, see
  \cite[Theorem 3.13]{Pitman}.
Consequently, \refE{Epower} applies with $\gam=1/\ga$ 
(after conditioning on $(q_i)$).
In particular, \as{} $v(\Gm)\asymp m^\ga$ and $e(\Gm)\asymp m^\ga\log m$ as
\mtoo.  

Moreover, $\Gm$ has \as{} a power-law degree distribution with exponent
$\tau=2$ in the sense of \refT{Tpower}.

Furthermore,  
\refT{Tgraphex} shows that the stretched empirical graphon converges \as{} 
in the sense
$\hwt\togp W$, where $W$ is the random graphon 
$W(x,y)=1-\exp\bigpar{-2Z^2x^{-1/\ga}y^{-1/\ga}}$ on $\bbRp$.
\end{example}

\begin{problem}
  In the simple graph Hollywood model with $0<\ga<1$ as in \refE{Ehollywood+},
does the degree distribution of $\Gm$ converge (\as, or at least
in probability) as \mtoo?
If so, what is the asymptotic distribution?
Is it random or deterministic?
\end{problem}

\section{Extremely sparse examples}\label{SXsparse}
\CCreset\ccreset

We can obtain extremely sparse examples in several ways.

First, \refT{Tdust} shows that any example including dust or attached stars is
extremely sparse.

Another way to obtain  extremely sparse graphs is to force the degrees to be
bounded, as follows.

\begin{example}
Let  $\mu=(\mu\xij)_{i,j=1}^\infty$ be a symmetric non-negative
matrix with  $0<\norm\mu<\infty$ and assume that each row contains at most
$d$ non-zero entries, for some $d<\infty$.
(For example, let  $\mu$ be a band matrix, with $\mu_{ij}=0$
unless $0<|i-j|\le d/2$.)

Since an edge $ij$ can exist only when $\mu_{ij}>0$, it follows that every
vertex in $\Gm$ has degree at most $d$. Hence the sequence $\Gm$ has bounded
degree, and in particular $\Gm$ is sparse; to be more precise we have
\begin{equation}
v(\Gm)\le 2e(\Gm) \le d v(\Gm).  
\end{equation}
\end{example}

Less obviously, it is also possible to obtain extremely sparse graphs in the
rank 1 case, with a sequence $q_i$ that decreases very slowly (remember that
$\sum_i q_i=1$ by assumption). We give one such example.

\begin{example}
  Consider the rank 1 case (\refS{Srank1simple}) with
$q_i=c/(i\log^2i)$ for $i\ge2$, where $c$ is the appropriate
normalization constant.
(Any $(q_i)$ with $q_i\asymp 1/(i\log^2i)$ would yield the same results
below.) Recall that, by comparison with an integral,
$\sum_{i\ge k} 1/(i\log^2i)\sim 1/\log k$ as $k\to\infty$.

For large $t$, let $\ellt:=\floor{t/\log^2t}$.
Then $\ellt\log^2\ellt\sim t$, and thus, using 
\eqref{vt1},
\begin{equation}\label{vett}
  v(t)\asymp \sum_{i=1}^\infty \bigpar{(q_it)\land1}
\asymp \sum_{i\le\ellt}1+\sum_{i>\ellt}\frac{t}{i\log^2i}
\asymp\ellt+\frac{t}{\log\ellt}
\asymp \frac{t}{\log t}.
\end{equation}
\end{example}

The expected number of edges is by \eqref{et1},
\begin{equation}\label{xet}
  \begin{split}
e(t)\asymp\sum_{i\neq j} \bigpar{(q_iq_jt)\land 1}    
\asymp\sum_{i\neq j} \frac{t}{i\log^2(i+1)\,j\log^2(j+1)}\land 1
  \end{split}
\end{equation}
We split the sum in \eqref{xet} into three (overlapping) parts.
The case $j\ge t^{0.4}$ yields at most
\begin{equation}\label{xet1}
  \begin{split}
    \sum_{i\ge1}\sum_{j\ge t^{0.4}} \frac{t}{i\log^2(i+1)\,j\log^2(j+1)}
\le \CC \frac{t}{\log t}.
  \end{split}
\end{equation}
The case $i\ge t^{0.4}$ yields the same, and finally the case 
$i,j<t^{0.4}$ yields at most
\begin{equation}\label{xet3}
  \begin{split}
    \sum_{i<t^{0.4}}\sum_{j< t^{0.4}}1 <t^{0.8}=o\Bigparfrac{t}{\log t}.
  \end{split}
\end{equation}
By \eqref{xet}--\eqref{xet3}, and the lower bound \eqref{e>v},
we find
\begin{equation}\label{evt}
  e(t)
\asymp \frac{t}{\log t}
\asymp v(t).
\end{equation}
Thus \refT{Too} yields $e(\Gt)\asymp v(\Gt)$ and $e(\Gm)\asymp v(\Gm)$ a.s.
In other words, $\Gm$ is extremely sparse.

We can be more precise.
Recall that $N_i(t)$ is the degree of vertex $i$ in $\GGt$; by \eqref{Ni}
$N_i(t)\sim \Po(\mu_i t)$.
Consequently, using also \eqref{mui1},
\begin{equation}
  \E \bigpar{N_i(t)\indic{N_i(t)>1}}
\le \E \bigpar{N_i(t)(N_i(t)-1)} =(\mu_it)^2\le 4q_i^2t^2.
\end{equation}
Summing over $i>t/\log^2t$ we obtain
\begin{equation}\label{berga}
\sum_{i>t/\log^2t}  \E \bigpar{N_i(t)\indic{N_i(t)>1}}
\le  \CC \sum_{i>t/\log^2 t} \frac{t^2}{i^2\log^4 i}
=O\Bigparfrac{t}{\log^2t}.
\end{equation}
Hence the expected number of edges that have one endpoint in 
$(t/\log^2t,\infty)$ and that endpoint is not isolated is $O(t/\log^2t)$.
Moreover, the expected number of edges with both endpoints in
$[1,t/\log^2t]$ is at most 
\begin{equation}\label{suntak}
  \begin{split}
  \sum_{i<j\le\tlogtt}2((tq_iq_j)\land1)
\le t^{0.8} +
 \sum_{t^{0.4}<j\le\tlogtt}\sum_{i<j}2((tq_iq_j)\land1)
  \end{split}
\end{equation}
where the last sum is at most a constant times, \cf{} \eqref{vett},
\begin{equation}\label{velinga}
  \begin{split}
 \sum_{j=t^{0.4}}^{\tlogtt} v(tq_j)
&\asymp
 \sum_{j=t^{0.4}}^{\tlogtt} \frac{tq_j}{\log(tq_j+2)}
\asymp  \sum_{j=t^{0.4}}^{\tlogtt} \frac{t/j\log^2t}{\log(t/j\log^2t+2)}
\\&
\asymp  
\int_{x=t^{0.4}}^{\tlogtt} \frac{t/\log^2t}{\log(t/x\log^2t+2)}\frac{\dd x}{x}
=   \int_{y=1}^{t^{0.6}/\log^2 t} \frac{t/\log^2t}{\log(y+2)}\frac{\dd y}{y}
\\&
=O\Bigpar{\frac{t}{\log^2t}\log\log t}
=o\Bigpar{\frac{t}{\log t}}.    
  \end{split}
\raisetag\baselineskip
\end{equation}
It follows by \eqref{berga}, \eqref{suntak} and \eqref{velinga} that, 
in $\GGt$ and thus in $\Gt$, all but
$\op(\tlogt)$ edges have one endpoint isolated. If the number
of such edges in $\Gt$ is $e'(\Gt)$, then thus the total number of edges is
$e(\Gt)=e'(\Gt)+\Op(\tlogtt)$, and since each edge has at most two
endpoints, the 
number of vertices is at least $e'(\Gt)$ and at most
$2e'(\Gt)+\Op(\tlogtt)$. 
Moreover, it is easily seen that the expected number of edges with both
endpoints in $(\tlogtt,\infty)$ is $O(\tlogtt)$, and it follows that, in
fact,
$v(\Gt)=e'(\Gt)+\Op(\tlogtt)$; we omit the details. 
Consequently, using also \eqref{tooe} and \eqref{evt},  
it follows that
$e(\Gt)/v(\Gt)\pto1$ as \ttoo. Moreover, we see also that almost
all edges belong to stars. (These are not attached stars in the sense of
\refS{Sgraphs}, since our example contains no attached stars, but they have
a similar effect on the graph.) As a consequence, at least in probability,
most vertices have degree 1, so the asymptotic degree distribution is
concentrated at 1.

Furthermore, a large fraction of the edges (and thus vertices) belong to a
finite number of such stars. To be precise, let $\eps>0$; then there exists
an integer $K=K(\eps)<\infty$ such that summing over $i>K$ only in \eqref{xet1}
yields $<\eps\tlogt$, which together with \eqref{xet3} and
\eqref{berga}--\eqref{velinga} shows that the expected number of edges that
are not in a star with centre at $i$ for some $i\le K$ is
$O\bigpar{\eps\tlogt}=O\bigpar{\eps e(t)}$. 

Since \as{} $\Gm\subseteq\tG_{2m}$ for all large $m$, the same results follow
also for $\Gm$.

Unfortunately, these properties make the random graphs in this  example
rather uninteresting for applications.

\section{Conclusions}\label{Sconclusion}

For the multigraph version, the examples in \refS{Srank1} seem very
interesting, but perhaps a bit special. We do not know whether they are
typical for a large class of interesting examples or not.

For the simple graph version, 
the examples above show that a great variety of different behaviour.
Nevertheless, the results are somewhat disappoining for applications; 
the relations between
the intensity matrix $(\mu\xij)$ and properties of the random graphs $\Gm$
such as edge density and degree distribution are far from obvious, and it is
not clear how one can choose the intensity matrix to obtain desired
properties; for example, we do not know any example of a power-law degree
distribution with an  exponent $\tau\neq2$.

Consequently, for both versions, it seems desirable to study more examples,
as well as to find more general theorems.

The present paper is only a first step (or rather second step, after
\cite{BroderickCai,CampbellCaiBroderick,CraneD-edge,CraneD-relational}), 
of the investigation of these random graphs, and it seems too early to tell
whether they will be useful as random graph models for various applications
or not.

\section*{Acknowledgement}
This work was mainly
carried out during a visit to the 
Isaac Newton Institute for Mathematical Sciences
during the programme 
Theoretical Foundations for Statistical Network Analysis in 2016
(EPSCR Grant Number EP/K032208/1)
and was partially supported by a grant from the Simons foundation, and
a grant from
the Knut and Alice Wallenberg Foundation.
I thank Harry Crane and Peter Orbanz
for helpful conversations at the Issac Newton Institute.

\newcommand\AAP{\emph{Adv. Appl. Probab.} }
\newcommand\JAP{\emph{J. Appl. Probab.} }
\newcommand\JAMS{\emph{J. \AMS} }
\newcommand\MAMS{\emph{Memoirs \AMS} }
\newcommand\PAMS{\emph{Proc. \AMS} }
\newcommand\TAMS{\emph{Trans. \AMS} }
\newcommand\AnnMS{\emph{Ann. Math. Statist.} }
\newcommand\AnnPr{\emph{Ann. Probab.} }
\newcommand\CPC{\emph{Combin. Probab. Comput.} }
\newcommand\JMAA{\emph{J. Math. Anal. Appl.} }
\newcommand\RSA{\emph{Random Struct. Alg.} }
\newcommand\ZW{\emph{Z. Wahrsch. Verw. Gebiete} }
\newcommand\DMTCS{\jour{Discr. Math. Theor. Comput. Sci.} }

\newcommand\AMS{Amer. Math. Soc.}
\newcommand\Springer{Springer-Verlag}
\newcommand\Wiley{Wiley}

\newcommand\vol{\textbf}
\newcommand\jour{\emph}
\newcommand\book{\emph}
\newcommand\inbook{\emph}
\def\no#1#2,{\unskip#2, no. #1,} 
\newcommand\toappear{\unskip, to appear}

\newcommand\arxiv[1]{\texttt{arXiv:#1}}
\newcommand\arXiv{\arxiv}

\def\nobibitem#1\par{}


\begin{thebibliography}{99}
\def\&{and}

\bibitem[Bollob\'as(2001)]{Bollobas}
B{\'e}la Bollob\'as.
\book{Random Graphs}. 2nd ed. Cambridge Univ. Press,
Cambridge, 2001.

\bibitem[Bollob\'as, Janson and Riordan(2007)]{SJ178}
B\'ela Bollob\'as, Svante Janson and Oliver Riordan.
The phase transition in inhomogeneous random graphs.
\RSA \vol{31} (2007), 3--122.


\bibitem[Borgs,  Chayes, Cohn and Holden(2016+)]{BCCH16}
Christian Borgs, Jennifer T. Chayes, Henry Cohn \& Nina Holden.
Sparse exchangeable graphs and their limits via graphon processes.
Preprint, 2016.
\arXiv{1601.07134v1}

\bibitem[Borgs, Chayes and Lov\'asz(2010)]{BCL:unique} 
Christian Borgs, Jennifer Chayes and  L{\'a}szl{\'o} \Lovasz.  
Moments of two-variable functions and the uniqueness of graph limits.  
\emph{Geom. Funct. Anal.}  \vol{19} (2010),  no. 6, 1597--1619.

\bibitem[Borgs, Chayes, \Lovasz, S\'os and Vesztergombi(2008)]{BCLSV1}
Christian Borgs, Jennifer T. Chayes, L{\'a}szl{\'o} \Lovasz, Vera T. S\'os
\& Katalin Vesztergombi. 
Convergent sequences of dense graphs I: Subgraph
frequencies, metric properties and testing,
\emph{Advances in Math.} {\bf 219} (2008), 1801--1851.

\bibitem[Borgs, Chayes, \Lovasz, S\'os and Vesztergombi(2012)]{BCLSV2}
Christian Borgs, Jennifer T. Chayes, L{\'a}szl{\'o} \Lovasz, Vera T. S\'os
\& Katalin Vesztergombi. 
Convergent sequences of dense graphs II. Multiway cuts and statistical
physics. 
\emph{Ann. of Math. (2)} \vol{176} (2012), no. 1, 151--219.

\bibitem[Broderick and Cai(2016)]{BroderickCai}
Tamara Broderick and Diana Cai.
Edge-exchangeable graphs and sparsity.
Preprint, 2016.
 \arXiv{1603.06898v1}


\bibitem[Campbell,  Cai and Broderick(2016)]{CampbellCaiBroderick}
    Trevor Campbell, Diana Cai and Tamara Broderick.
    Exchangeable trait allocations.
Preprint, 2016.
\arXiv{1609.09147v1}

\bibitem[Caron and Fox(2014+)]{CaronFox}
Fran{\c c}ois Caron \& Emily B. Fox.
Sparse graphs using exchangeable random measures.
Preprint, 2014.
\arXiv{1401.1137v3}

\bibitem{Crane-Ewens} 
Harry Crane. 
The ubiquitous Ewens sampling formula. 
\emph{Statist. Sci.} \textbf{31} (2016), no. 1, 1--19. 

\bibitem[Crane and Dempsey(2016)]{CraneD-edge}
Harry Crane and Walter Dempsey.
Edge  exchangeable models for network data.
Preprint, 2016.
\arxiv{1603.04571v3}

\bibitem[Crane and Dempsey(2016)]{CraneD-relational}
Harry Crane and Walter Dempsey.
Relational exchangeability.
Preprint, 2016.
\arxiv{1607.06762v1}

\bibitem{Darling67} 
Donald A. Darling. 
Some limit theorems associated with multinomial trials. 
\emph{Proc. Fifth Berkeley Sympos. Math. Statist. and Probability
  (Berkeley, Calif., 1965/66)}, 
Vol. II: Contributions to Probability Theory, Part 1, pp. 345--350,
Univ. California Press, Berkeley, CA,  1967.


\bibitem[Diaconis and Janson(2008)]{SJ209}
Percy Diaconis and Svante Janson.
Graph limits and exchangeable random graphs.
\emph{Rendiconti di Matematica}
\textbf{28} (2008), 33--61.

\bibitem{Dutko}
Michael Dutko.
Central limit theorems for infinite urn models.
\emph{Ann. Probab.} \textbf{17} (1989), no. 3, 1255--1263. 

\bibitem[Erd\H{o}s and R\'enyi(1960)]{ER1960}
Paul Erd\H{o}s and Alfr{\'e}d R\'enyi. 
\newblock On the evolution of random graphs. 
\newblock {\em Magyar Tud.\ Akad.\ Mat.\ Kutat\'o Int.\ K\"ozl} 
{\bf 5} (1960), 17--61.


\bibitem[Herlau,  Schmidt and M\o rup(2015+)]{HSM15}
Tue Herlau, Mikkel N. Schmidt \& Morten M\o rup, 
Completely random measures for modelling block-structured networks.
Preprint, 2015.
\arXiv{1507.02925v3}

\bibitem{RemcoI}
Remco van der Hofstad.
\emph{Random Graphs and Complex Networks, Volume 1}.
Cambridge University Press, Cambridge, 2017.

\bibitem{SJ189}
Hsien-Kuei Hwang and Svante Janson.
Local limit theorems for finite and infinite urn models.
\emph{Ann. Probab.} \textbf{36} (2008), no. 3, 992--1022. 


\bibitem[Janson(2004)]{SJ154}  
Svante Janson, 
Functional limit theorems for multitype branching processes and generalized
\Polya{} urns.  
\emph{Stoch. Process. Appl.} \textbf{110} (2004),  177--245.

\bibitem[Janson(2005)]{SJ169}
Svante Janson,
Limit theorems for triangular urn schemes. 
\emph{Probab. Theory  Rel. Fields} \vol{134} (2005), 417--452. 


\bibitem[Janson(2013)]{SJ249}
Svante Janson.
Graphons, cut norm and distance, rearrangements and coupling.
\emph{New York J. Math. Monographs}
\textbf4, 2013.

\bibitem[Janson(2016)]{SJ311}
Svante Janson.
Graphons and cut metric on $\sigma$-finite measure spaces.
Preprint, 2016.
\arxiv{1608.01833v1}

\bibitem[Janson(2017)]{SJ317}
Svante Janson.
On convergence for graphexes.
Preprint, 2017.


\bibitem{SJ97} 
Svante Janson, Donald E. Knuth, Tomasz \L uczak \& Boris Pittel.
 The birth of the giant component.
\RSA
 \textbf{4} (1994), 231--358.

\bibitem[Janson, {\L}uczak and Ruci\'nski(2000)]{JLR}
Svante Janson, Tomasz \L uczak \& Andrzej Ruci\'nski.
\book{Random Graphs}.
\Wiley, New York, 2000.

\bibitem{SJ-Warnke}
Svante Janson and Lutz Warnke.
In preparation.


\bibitem{Jirina}
Miloslav Ji\v rina,  
Stochastic branching processes with continuous state space.  
\jour{Czechoslovak Math. J.} \vol{8 (83)}  (1958), 292--313.

\bibitem{JohnsonKotz}
Norman L. Johnson \& Samuel Kotz,
\emph{Urn Models and their Application}.
Wiley, New York, 
1977. 


\bibitem{Karlin}
Samuel Karlin. 
Central limit theorems for certain infinite urn schemes. 
\emph{J. Math. Mech.} \textbf{17} (1967), 37--401.

\bibitem{KestenMR} 
Harry Kesten (1968). 
Review of Darling, Some limit
theorems associated with multinomial trials. 
\emph{Math. Reviews} 
\textbf{35} \#7378, MR0216547.

\bibitem[\Lovasz{}(2012)]{Lovasz}
L{\'a}szl{\'o} Lov{\'a}sz,
\emph{Large Networks and Graph Limits}.
American Mathematical Society, Providence, RI, 2012.

\bibitem[Markov(1917)]{Markov1917} 
A. A. Markov,
Sur quelques formules limites du calcul des probabilit\'es
(Russian).
\emph{Bulletin de l'Acad\'emie Imp\'eriale des Sciences, Petrograd}
\textbf{11} (1917), no. 3, 177--186.


\bibitem[Orbanz and Roy(2015)]{OR}  
Peter Orbanz \& Daniel M. Roy.
Bayesian models of graphs, arrays and other exchangeable structures.
\emph{IEEE Trans. Pattern Analysis and Machine Intelligence} \textbf{37}
(2015), no. 2, 437--461.

\bibitem{Pitman}
Jim Pitman. 
\emph{Combinatorial Stochastic Processes}. 
Ecole d'Et\'e de Probabilit\'es de Saint-Flour XXXII -- 2002.
Lecture Notes in Mathematics, 1875. Springer-Verlag, Berlin, 2006. 

\bibitem[Pittel(2009)]{Pittel}
Boris Pittel.
On a random graph evolving by degrees. 
\emph{Adv. Math.} \textbf{223} (2010), no. 2, 619--671. 

\bibitem[P\'olya(1931)] {Polya}
G. P\'olya, 
Sur quelques points de la th\'eorie des probabilit\'es.
\jour{Ann. Inst. Poincar\'e} \vol1 (1930),
117--161.

\bibitem{Smythe}
R. T. Smythe.
Strong laws of large numbers for $r$-dimensional arrays of random variables.
\emph{Ann. Probability} \textbf1 (1973), no. 1, 164--170. 

\bibitem[Veitch and Roy(2015+)]{VR}
Victor Veitch \& Daniel M. Roy.
The class of random graphs arising from exchangeable random measures.
Preprint, 2015.
\arxiv{1512.03099}

\bibitem[Veitch and Roy(2016+)]{VR2}
Victor Veitch \& Daniel M. Roy.
Sampling and estimation for (sparse) exchangeable graphs.
Preprint, 2016.
\url{arXiv:1611.00843v1} 

\end{thebibliography}
\end{document}